\documentclass{article}
\usepackage[utf8]{inputenc}
\usepackage[english]{babel}
\usepackage{setspace}
\usepackage{appendix}
\usepackage{hyperref}
\usepackage{natbib}
\usepackage{graphicx}
\usepackage{pgfplots}
\usepackage{float}
\usepackage{amsmath}
\usepackage{amsthm}
\usepackage{amssymb}
\usepackage{mathrsfs}
\usepackage{dsfont}
\usepackage{authblk}
\usepackage{thmtools}
\usepackage{thm-restate}
\usepackage{tikz}
\usetikzlibrary{positioning,arrows,arrows.meta,shapes.geometric,fit,backgrounds}
\usepackage{xcolor}

\usepackage[a4paper, left=1.2in, right=1.2in, top=1.2in, bottom=1.2in]{geometry}

\usetikzlibrary{arrows,calc,shapes,decorations.pathreplacing,patterns}
\usepackage{pgfplots}
\pgfplotsset{compat=1.18}
\usepackage{amsfonts}

\newtheorem{theorem}{Theorem}

\newtheorem{coro}{Corollary}

\newtheorem{lemma}{Lemma}
\newtheorem{definition}
{Definition}
\newtheorem{remark}{Remark}
\newcommand{\dmls}{p^*_\xi}
\newcommand{\mlsfun}[1]{s^\text{MLS}_{f,X}(#1)}
\newcommand{\mlsfunstoch}{s^\text{MLS}_{F_n,X_n}}

\title{Moving Least Squares without Quasi-Uniformity: A Stochastic Approach
}

\author{
Shir Tapiro Moshe\thanks{Department of Statistics and Data Science, The Hebrew University of Jerusalem, Mount Scopus, Jerusalem 91905, Israel. Email: \texttt{shir.tapiro@mail.huji.ac.il}. \textbf{Corresponding author.}}
\and
Yariv Aizenbud\thanks{Department of Applied Mathematics, School of Mathematical Sciences, Tel Aviv University, Tel Aviv 69978, Israel. Email: \texttt{aizeny@tauex.tau.ac.il}}
\and
Barak Sober\thanks{Department of Statistics and Data Science, Center for Digital Humanities, The Hebrew University of Jerusalem, Mount Scopus, Jerusalem 91905, Israel. Email: \texttt{barak.sober@mail.huji.ac.il}}
}
\begin{document}
\date{}
\maketitle

\begin{abstract}
Local Polynomial Regression (LPR) and Moving Least Squares (MLS) are closely related nonparametric estimation methods, developed independently in statistics and approximation theory. While statistical LPR analysis focuses on overcoming sampling noise under probabilistic assumptions, the deterministic MLS theory studies smoothness properties and convergence rates with respect to the \textit{fill distance} (a resolution parameter). Despite this similarity, the deterministic assumptions underlying MLS fail to hold under random sampling.
We begin by quantifying the probabilistic behavior of the fill distance $h_n$ and \textit{separation} $\delta_n$ of an i.i.d. random sample. 
That is, for a distribution satisfying a mild regularity condition, $h_n\propto n^{-1/d}\log^{1/d} (n)$ and $\delta_n \propto n^{-2/d}$ in probability.
We then prove that, for MLS of degree $k\!-\!1$, the approximation error associated with a differential operator $Q$ of order $m\le k-1$ decays as $h_n^{\,k-m}$, establishing stochastic analogues of the classical MLS estimates.
Additionally, we show that the MLS approximant is locally smooth with high probability.
This work provides the first unified stochastic analysis of MLS, demonstrating that---despite the failure of deterministic sampling assumptions---the classical convergence and smoothness properties persist under natural probabilistic models.
\end{abstract}


\textup{(Primary)}; 41A10; 62G20; 65D15; 60D05 \textup{(Secondary)}.

\section{Introduction}

Local Polynomial Regression (LPR) and Moving Least Squares (MLS) are two closely related nonparametric methods, both relying on weighted local polynomial fitting. While they emerged independently—LPR from statistics and MLS from numerical analysis—their local formulations are algebraically equivalent. This parallel development has led to two complementary traditions: one focused on statistical inference under random sampling, the other on numerical approximation under deterministic sampling.

In statistics, LPR was introduced as a framework for estimating functions and their derivatives under random designs. The method constructs a local polynomial approximation around a target point, with nearby data weighted more heavily than distant ones. Traditional developments in this field—ranging from foundational works by Stone \cite{stone1982optimal} and Tsybakov \cite{tsybakov2008introduction} to refined asymptotic analyses by Ruppert and Wand \cite{ruppert1994multivariate}—primarily focus on handling observational noise and analyzing the resulting estimators for point or interval estimation. Crucially, these classical statistical analyses evaluate the derivatives of the local fitted polynomial at a single, fixed target point. Furthermore, because their primary objective is managing measurement noise, the underlying geometric stability of the random point configurations is often secondary or handled via simplifying distributional assumptions.

In parallel, MLS was developed in computational mathematics as a meshfree method for function approximation \cite{levin1998approximation}. Like LPR, it employs weighted polynomial fitting, but its construction yields smooth shape functions with local support that enable continuous reconstruction across a domain. This property makes MLS especially suitable for geometric applications, such as surface reconstruction \cite{levin2004mesh}, solving partial differential equations \cite{mirzaei2015analysis}, and manifold approximation \cite{sober2020manifold, sharon2023multiscale}. The deterministic analyses of MLS rely on geometric sampling conditions, notably the notions of \textit{fill distance} (sometimes referred to as the \textit{covering radius}), the largest distance from any point of the domain to its nearest sample, and \textit{separation}, the minimal distance between two distinct points. In this setting, the two quantities are comparable up to a constant factor, which provides a convenient framework for proving convergence and stability. Within this setting, Levin \cite{levin1998approximation} demonstrated convergence rates for smooth functions, while Mirzaei \cite{mirzaei2015analysis} extended the theory to derivatives of functions. In particular, Mirzaei showed that if one applies a differential operator of order $m$ to both the MLS approximation and the true function, the difference between them decays like the fill distance raised to the power $k-|m|$, where $k$ is the smoothness order of the function. These results underscore the strength of MLS in preserving smoothness and geometric structure under deterministic sampling assumptions.
More recently, Krieg and Sonnleitner \cite{krieg2024random} relaxed the reliance on fill distance by showing that random points on convex domains are asymptotically optimal for approximating functions in Sobolev spaces. This suggests that the large gaps characteristic of stochastic sampling do not prevent optimal average convergence rates.

Although LPR and MLS share an algebraic equivalence in their local formulations, they diverge significantly in their analytical objectives and geometric assumptions. In this work, we operate in a purely noiseless approximation framework—where observational noise is completely absent—and adopt the stochastic language of random designs solely to model the random spatial layout of the sample points. In this stochastic setting, a fundamental mathematical obstacle arises: the deterministic assumptions underlying classical MLS theory completely collapse. Under random sampling, the fill distance and the separation distance are no longer comparable; instead, random points form dense local clusters that threaten the algebraic stability and invertibility of the local matrices.

The main contribution of this work is to provide a rigorous stochastic analysis that overcomes this geometric stability bottleneck. Unlike the classical LPR literature, which evaluates the point-wise derivatives of a separate local polynomial fit at each target point, our work establishes convergence guarantees for the derivatives of the MLS approximation function itself—a single function defined across the domain rather than a pointwise local fit. This functional behavior is a critical prerequisite for practical applications, such as using meshless methods to solve partial differential equations (PDEs) \cite{mirzaei2015analysis}. By showing that the local design matrices remain stable despite the breakdown of quasi-uniformity, we successfully extend the deterministic derivative approximation results of Mirzaei \cite{mirzaei2015analysis} to purely random sampling configurations.

\section{Preliminaries}

We first recall some standard geometric concepts from the deterministic MLS framework.
\begin{definition}
        Let $\Omega$ be a domain in $\mathbb{R}^d$ , and consider sets of data points in $\Omega$. We say that $X=\left\{x_i\right\}_{i=1}^I$ is a $h-\delta$ set if:
        \begin{enumerate}
            \item $h$ is the "fill distance" with respect to the domain $\Omega$, 
            \begin{equation*}
                h=\sup _{x \in \Omega} \min _{x_i \in X}\left\|x-x_i\right\|,
            \end{equation*}
            \item The "separation" is the largest $\delta$ such that 
            \begin{equation*}
                \left\|x_i-x_j\right\| \geq  \delta, \quad \forall \, 1 \leq i<j \leq I.
            \end{equation*}
        \end{enumerate}
\end{definition}
The next definition describes a condition on the sample points that ensures they are neither too close together nor too far apart, providing a foundation for all subsequent deterministic MLS results.
\begin{definition}
        (quasi-uniform sample set). A set of data sites $X=\left\{x_1, \ldots, x_n\right\}$ is said to be quasi-uniform with respect to a domain $\Omega$ and a constant $c_{\mathrm{qu}} > 0$ if
        $$\delta \leq h \leq c_{\mathrm{qu}} \delta $$
        where $h$ is the fill distance and $\delta$ is the separation.
        \label{h-delta-condition -levin}
\end{definition}  

While this definition serves as the foundation for the subsequent results in the deterministic setting, it does not hold in the stochastic setup, which required us adapt the proofs to the stochastic conditions, as will be seen later.\\
We now recall some central results from the deterministic MLS framework, focusing on the construction of the local polynomial approximation and its convergence properties.

Let $\theta_h \in \mathcal{C}^k$ (continuously differentiable $k$ times over $\Omega$) be a weight function with compact support of size $s\cdot h$ of the shape
\begin{equation}
 \theta_h(x)=\Phi\left(\frac{x}{h}\right),
 \label{theta_weight func}
\end{equation}
where $\Phi: \mathbb{R}^d \rightarrow \mathbb{R}$ is a nonnegative function supported in the ball $B(0,s)$ (the ball centered at the origin with radius $s$), where $\Phi(x)\geq 0$ for all $x\in B(0,s)$.

Let $\left\{x_i\right\}_{i=1}^n $ be a set of distinct scattered points in $\mathbb{R}^d$, and $\left\{f\left(x_i\right)\right\}_{i=1}^n$ be the corresponding sampled values of some function $f: \mathbb{R}^d \rightarrow \mathbb{R}$. Then, we define the local polynomial by $\dmls$, 
      \begin{equation}
      \dmls=\underset{p \in \Pi_{k-1}^d}{\operatorname{argmin}} \sum_{i=1}^n\left(p\left(x_i\right)-f\left(x_i\right)\right)^2 \theta_h\left(\xi-x_i\right),
          \label{dmls}
      \end{equation} 
where $\Pi_{k-1}^d $ is the space of polynomials of total degree $k-1$ in $\mathbb{R}^d$. The MLS function is defined by 
 $$ \mlsfun{x}:=p^*_x(0).$$
Based on the above, Levin \cite{levin1998approximation} showed that if $f\in \mathcal{C}^k(\Omega)$ then      $$\|\mlsfun{x}-f\|_{\Omega, \infty}<L \cdot h^{k},$$
where $\| \mlsfun x -f\|_{\Omega, \infty}=\sup_{x \in \Omega}| \mlsfun x-f(x)|$, and $L$ is a constant.
Another significant theoretical result was established in Mirzaei's work \cite{mirzaei2015analysis}, which demonstrated that for some constant $C$ independent of $f$ and $h$, if $f$ is $k$-times differentiable, then  
$$
\left|\partial^\alpha \mlsfun{\hat{x}}-\partial^\alpha f(\hat{x})\right|\leq C|f|_{\mathcal{C}^k(\hat{x})} h^{k-|\alpha|},
$$
where $\alpha=\left(\alpha_1, \ldots, \alpha_d\right) $ is a multi-index, meaning an ordered tuple of nonnegative integers used to denote partial derivatives:
$ \partial^\alpha f=\frac{\partial^{|\alpha|} f}{\partial x_1^{\alpha_1} \ldots \partial x_d^{\alpha_d}} \text {, where }|\alpha|=\alpha_1+\cdots+\alpha_d$, and the seminorm, $|f|_{\mathcal{C}^k\left(B\left(\hat{x},s h\right)\right)} $, is defined by
$$
|f|_{\mathcal{C}^k\left(B\left(\hat{x},s h\right)\right)} \triangleq \max _{\substack{|\alpha|=k \\ x \in B\left(\hat{x}, s h\right)}}\left|\partial^\alpha f(x)\right|
,$$
More generally, let \( Q \) be a linear differential operator with constant coefficients, given by  
$$Q=\sum_{|\alpha| \leq k} q_\alpha \partial^\alpha,$$
where \( q_\alpha \) are real constants for \( |\alpha| \leq k \). Here, $\alpha=\left(\alpha_1, \ldots, \alpha_d\right)$ denotes a multi-index of nonnegative integers. Its order is $|\alpha|=\alpha_1+\cdots+\alpha_d$, and for a smooth function $f$ we write $\partial^\alpha f=\frac{\partial^{|\alpha|} f}{\partial x_1^{\alpha_1} \cdots \partial x_d^{\alpha_d}}
$.
Let \( m \) denote the order of \( Q \), defined by  
$$m=\max \left(|\alpha|:|\alpha| \leq k-1 \text { and } q_\alpha \neq 0\right).$$
(For example, if \( Q f = f \), then \( m=0 \); while if \( Q f = \partial^2 f / \partial x_1^2+\cdots+\partial^2 f / \partial x_1\partial x_2 \), then \( m=2 \)). 

With these definitions and results in place, we are now ready to extend the MLS framework from the deterministic setting to stochastic samples, which will be the focus of the next section.

\section{Main Results}

\begin{theorem}
    For a distribution $\mathbb{P}$ over $\Omega$ that satisfies Definition \ref{NicelyDef}, let $X_n=\{x_1,x_2,...,x_n\}$ be a set of $n$ $iid$ samples of $\mathbb{P}$. Denote by $h_n$ the fill distance of $X_n$ (with respect to $\Omega$), and by $\delta_n$ the separation of $X_n$. 
    Then, 
         \begin{equation}
            h_n = O_P(n^{-\frac{1}{d}} \ln ^{\frac{1}{d}}(n))\quad \mbox{and} \quad \delta_n = \Omega_P(n^{-\frac{2}{d}}).
         \end{equation}
         \label{theorem of order of probability}
\end{theorem}
\begin{proof}
    The proof follows immediately from Lemma \ref{Order in probability of h_n} and Lemma \ref{Order in probability of delta_n} below.
\end{proof}

As a result, we needed to adapt the proofs of Mirzaei \cite{mirzaei2015analysis} so that they hold under the stochastic conditions of Theorem \ref{theorem of order of probability}. Ultimately, we are able to establish the following theorem, which serves as the stochastic analogue of Mirzaei’s result.

\begin{theorem}
Let $X_n$ be a sample set of size $n$ drawn from a nicely behaving distribution \emph{(}Definition \ref{NicelyDef}\emph{)} on $\Omega$, which satisfies the interior cone condition \emph{(}Definition \ref{cone_cond}\emph{)}. 
Let $F_n$ be the corresponding sampled values of some function $f:\mathbb{R}^d\rightarrow\mathbb{R}$, and let $\mlsfunstoch(x)$ be the value of the $(k-1)$ degree MLS function at the point $x$, then   for $\hat{x}\in \Omega$ 

\begin{equation}
P\!\left(\left|Q\mlsfunstoch(\hat{x})-Qf(\hat{x})\right|
\ge \mathcal{K}\,h_n^{k-m}\right)\leq \frac{C_{M_Q}}{n^{c_\mu}}
\label{mainResultEqProb}
\end{equation}
i.e.,
\begin{equation}
\lim_{n\rightarrow\infty}P\left(\left|Q \mlsfunstoch(\hat{x})-Qf(\hat{x})\right| \geq\mathcal{K}\cdot h_n^{k-m}\right)=0,
\end{equation}
where $\mathcal{K},C_{M_Q},{c_\mu}$ are some positive constants independent of $n$, $Q$ is a linear differential operator of degree $m< k$.    
    \label{main_result}
\end{theorem}

See proof in Section \ref{Main_Result_proof}.

We also show that , the MLS approximation is smooth, similarly to the deterministic scenario shown in \cite{levin1998approximation}:

\begin{theorem}

[Local smoothness of the MLS approximation]\label{thm:local-smoothness}
Given the event $\mathscr{E}_I$ hold \eqref{E_I event} and $\theta_h \in \mathcal{C}^k(\Omega)$. For every fixed $x_0 \in \Omega$ there exists a
radius $\delta_0 \asymp h_n$ such that the MLS approximation
$s^{\mathrm{MLS}}_{F_n,X_n}(x)$ is $k$-times smooth on the ball
$B(x_0,\delta_0)$, with probability at least $1 - \tfrac{2A^2}{n^{\phi}}$, where
$\phi$ is a positive constant. Namely,
\[
  P\!\left( s^{\mathrm{MLS}}_{F_n,X_n}(x) \in \mathcal{C}^k\!\big(B(x_0,\delta_0)\big)\Big|\mathscr{E}_I \right)
  \;\ge\; 1 - \frac{2A^2}{n^{\phi}}.
\]
\label{theorem of smoothness}
\end{theorem}
See proof in Section \ref{smoothness of the approximation}.

Theorem~\ref{thm:local-smoothness} and Corollary~\ref{cor:local-shapefun}
are local: they guarantee smoothness on a ball $B(x_0,\delta_0)$ whose radius
$\delta_0 \asymp h_n$ is dictated by the bandwidth. A global statement on all
of $\Omega$ is attainable by enlarging the bandwidth, but only at the expense
of the approximation rate; since our sampling is noise-free, this trade-off is
a genuine loss, and we therefore retain the small, fill-distance--scaled
bandwidth and a local result. We expand on this point in the remarks following
Corollary~\ref{cor:local-shapefun}.  

\section{Bridging the gap between deterministic and stochastic MLS}

This section lays the probabilistic foundations required to rigorously compare deterministic and stochastic MLS frameworks. We begin by introducing two central definitions: the interior cone condition, which imposes a geometric regularity on the domain, and a family of distributions with well-behaved sampling properties. Under these assumptions, we establish that the convergence rate of the fill distance differs from that of the separation distance—a key distinction from the deterministic setting. We then derive properties of how the number of sample points in local neighborhoods grows with 
$n$, and other structural properties of the random point cloud, such as local point density and inter-point distances, that are essential for the analysis.

\begin{definition}
    A set $\Omega\subseteq\mathbb{R}^d$ is said to satisfy an \textbf{interior cone condition} if there exists an angle $\theta\in (0,\pi/2)$ and a radius $r > 0$ such that for every $x\in \Omega$ there is a unit vector $\xi(x)$, such that the cone
    \begin{equation}
        C(x, \xi(x), \theta, r):=\left\{x+\lambda y: y \in \mathbb{R}^d,\|y\|_2=1, y^T \xi(x) \geq \cos \theta, \lambda \in[0, r]\right\}
    \end{equation}
    is contained in $\Omega$.
    \label{cone_cond}
\end{definition}

\begin{remark}
    If the domain has no boundary, then the interior cone condition (see Definition \ref{cone_cond}) is automatically satisfied, and all proofs in the article remain valid. However, if the domain has a boundary, the interior cone condition becomes necessary. The interior cone condition ensures the existence of an angle $\theta$ and a radius $r$ such that a cone of positive volume can be positioned at any point in the domain $\Omega$. As $n \to \infty$, the number of data points contained within such a cone tends to infinity in probability. In particular, the expected number of samples in the cone is proportional to $n$, with proportionality factor bounded below by a positive constant depending only on $\theta$, $r$, and $d$. This guarantees that no substantial difficulty arises near the boundary.
\end{remark}

Having introduced the geometric assumptions on the domain, we now formalize the probabilistic setting.
Let $(\Omega,p(x))$ be a probability space. 
Assume $\Omega$ is a closed compact and satisfies an interior cone condition. Let $p(x)$ be the probability density function. We assume $p(x)>0$ $\forall x\in \Omega$.

\begin{definition} ("nicely behaving distribution").
    Let $\mu$ denote the uniform measure over $\Omega$, induced by the Lebesgue measure.
A probability density function $p(x)$ on $\Omega$ is called nicely behaving distribution, if there exist constants $0<\underline{c} \leq \bar{c}<\infty$ such that $\underline{c} \cdot \mu \leq p(x) \leq \bar{c} \cdot \mu$.
    \label{NicelyDef}
\end{definition}

We now turn to the proof of the Theorem \ref{theorem of order of probability} stated above. The proof is divided into two lemmas: the first establishes the asymptotic behavior of the fill distance, while the second addresses the asymptotic behavior of the separation. Together, these results complete the proof of the theorem and highlight the fundamental distinction between the stochastic and deterministic settings. Specifically, while the deterministic framework typically assumes quasi-uniformity (see Definition \ref{h-delta-condition -levin}), this property fails in the stochastic case.

\begin{lemma}
For a distribution $\mathbb{P}$ over $\Omega$ that satisfies Definition \ref{NicelyDef}, let $X_n=\{x_1,x_2,...,x_n\}$ be a set of $n$ $iid$ samples of $\mathbb{P}$. Denote by $h_n$ the fill distance with respect to $X_n$. Then, 
         \begin{equation}
            h_n = O_P(n^{-\frac{1}{d}} \ln ^{\frac{1}{d}}(n)).
         \end{equation}
    \label{Order in probability of h_n}
\end{lemma} 

\begin{remark}
    A similar upper bound for $h_n$ was first established in \cite{reznikov2016covering}. We include this proof for completeness.
\end{remark}

\begin{proof}
    From the interior cone condition we obtain that $ \text{Vol}_d(B(x,\varepsilon)\cap \Omega) >0$ for all $x\in \Omega$ and all real $\varepsilon>0$, 
    where $B(x,\varepsilon)$ is the open ball in $\mathbb{R}^d$ of radius $\varepsilon$ centered at $x$. Then the condition $p(x)>0$ for all $x\in \Omega$ implies that $P(x_1\in B(x,\varepsilon)\cap \Omega)>0$ and hence $P(x_1\notin B(x,\varepsilon)\cap \Omega)<1$, for all $x\in \Omega$ and all real $\varepsilon>0$. 
    In addition, by the compactness of $\Omega$, for any real $\varepsilon>0$ there is a finite minimal covering set $F_\varepsilon\subseteq \Omega$ such that $\Omega\subseteq \bigcup_{x\in F_\varepsilon}B(x,\varepsilon)$. The cardinality of $F_\varepsilon$ represents the smallest number of closed balls with centers in $\Omega$ and radii $\varepsilon$ whose union covers $\Omega$. We denote the covering number by $\mathcal{N}_\varepsilon = |F_\varepsilon|$. 
    Then, for each real $\varepsilon>0$,
    \begin{equation}
        P\left(h_n \geq 2 \varepsilon\right) \leq P\left(\bigcup_{x \in F_{\varepsilon}} \bigcap_{i=1}^n\left\{x_i \notin B(x,\varepsilon)\right\}\right) \leq \sum_{x \in F_{\varepsilon}} P\left(x_1 \notin B(x,\varepsilon)\right)^n.
        \label{gamma>2eps}
    \end{equation}
Note that as $n\rightarrow \infty$ this probabilty goes to zero, hence, $h_n\rightarrow 0$ in probability. \\
Now, in order to find the probabilistic rate of decay of $h_n$, we first bound the covering number of the domain.
It is shown in Proposition 4.2.12 of \cite{vershynin2018high} that for any set $K \subset \mathbb{R}^d$ and any $\varepsilon > 0$, the covering number $\mathcal{N}_\varepsilon$ satisfies
\begin{equation}
    \frac{\operatorname{vol}_d(K)}{\operatorname{vol}_d(B(x, \varepsilon))} \leq \mathcal{N}_\varepsilon \leq \frac{\operatorname{vol}_d(K + B(x, \varepsilon))}{\operatorname{vol}_d(B(x, \varepsilon))},
    \label{Vershynin_prop}
\end{equation}
where $\operatorname{vol}_d(\cdot)$ denotes $d$-dimensional Lebesgue measure and $+$ denotes the Minkowski sum.
Now, if we denote by $P( B(x,\varepsilon))$ the probability to have a sample in $B(x,\varepsilon)$, then according to the assumptions $\int_{B(x,\varepsilon)} \underline{c} d \mu \leq \int_{B(x,\varepsilon)} d p \leq\int_{B(x,\varepsilon)} \bar{c} d \mu$, it follows that $P(B(x,\varepsilon))\geq \underline{c}\cdot P_{U}(B(x,\varepsilon))$, where $P_U$ is the probability density function (pdf) of the uniform distribution over $\Omega$. 
Thus, the probability that there are no samples in $B(x,\varepsilon)$, is less or equal $ 1- \underline{c}\cdot P_{U}(B(x,\varepsilon))$,
    and therefore we get that 
    \begin{equation}
        \sum_{x \in F_{\varepsilon}} P\left(x_1 \notin B(x,\varepsilon)\right)^n \leq \sum_{i=1}^{\mathcal{N}_\varepsilon}\left(1-\underline{c}\cdot P_{U}(B(x,\varepsilon))\right)^n =\mathcal{N}_\varepsilon\left(1-\underline{c}\cdot P_{U}(B(x,\varepsilon))\right)^n.
    \end{equation}
(Note that $h_n\rightarrow0$ as $n\rightarrow\infty$).
Since $\Omega$ is compact, its volume exists and is finite. In addition, for a very small $\varepsilon$, the contribution of $B(x,\varepsilon)$ to the total volume is marginal, and we denote this total volume by $vol_d(\Omega +B(x,\varepsilon))=\kappa$.  Using this notation with substituting Equation \eqref{Vershynin_prop}, we get that 
    \begin{equation}
        P\left(h_n \geq 2 \varepsilon\right) \leq \frac{\kappa}{\varepsilon ^d}\left(1-\underline{c}\cdot P_{U}(B(x,\varepsilon))\right)^n. \end{equation}
    In particular, we have that $P_{U}(B(x,\varepsilon))=C\cdot\varepsilon^d$, where $C=\frac{\pi^\frac{d}{2}}{\text{Vol}_d(\Omega)\Gamma(\frac{d}{2}+1)}$ and $\Gamma$ is Euler's Gamma function. Hence, if $c=\underline{c}\cdot C$, then
    \begin{equation}
        P\left(h_n \geq 2 \varepsilon\right) \leq \frac{\kappa}{\varepsilon ^d}\left(1-c \varepsilon^d\right)^n. 
        \label{gamma>2epsFinal}
    \end{equation}
Now, in order to find convergence rate of $h_n$, we need to express $\varepsilon$ in terms of n. We know that $1 - x \leq e^{-x}$ for all $x$. So, if we let $x = c\varepsilon^d$, then $$1 - c\varepsilon^d \leq e^{-c \varepsilon^d}.$$
Using this inequality in Equation (\ref{gamma>2epsFinal}), we find that
\begin{equation*}
     \frac{\kappa}{\varepsilon ^d} \cdot\left(1-c \varepsilon^d\right)^n \leq\frac{\kappa}{\varepsilon ^d}\cdot e^{-n c \varepsilon^d}.
\end{equation*}
Furthermore, applying $e^{-y}\leq y\cdot e^{-y}$ for $y\geq 1$, then for $y=nc\varepsilon^d$ we find that
\begin{equation*}
    \frac{\kappa}{\varepsilon ^d}\left(e^{-nc \varepsilon^d}\right) \leq\frac{\kappa}{\varepsilon ^d} \cdot \left(n c\varepsilon^d\cdot e^{-nc \varepsilon^d}\right).
\end{equation*} 
So now we require that $\frac{\kappa}{\varepsilon ^d} \cdot \left(n c\varepsilon^d\cdot e^{-nc\varepsilon^d}\right)\leq\rho_h$ (where $0<\rho_h<1$ since it represents a probability), and if we solve the later inequality we find that 
\begin{equation}
\varepsilon\geq\frac{\ln^\frac{1}{d}(n \cdot \frac{\kappa c}{\rho_h})}{n^\frac{1}{d}\cdot c^\frac{1}{d}}.
    \label{eps.n.d}
\end{equation}
For any  $d\geq 1$, exists $K_h$, a real positive constant, and $N_h$ such that for any $n\geq N_h$,

\begin{equation}
    \varepsilon \geq K_h\cdot n^{-\frac{1}{d}}\cdot \ln ^\frac{1}{d}(n).
    \label{lower bound eps}
\end{equation}
We can then rewrite Equation (\ref{gamma>2epsFinal}) as
\begin{equation*}
    P(h_n\geq 2 \varepsilon)\leq\rho_h,
\end{equation*}
which is equivalent to
\begin{equation*}
    P\left(\left|\frac{h_n}{\varepsilon}\right|>2\right)\leq\rho_h.
\end{equation*}
Substituting the lower bound of $\varepsilon$ from Equation \eqref{lower bound eps}, keeps this inequality valid. Hence, exists a finite $N_h>0$ such that
\begin{equation*}
    P\left(\left|\frac{h_n}{K_h\cdot n^{-\frac{1}{d}}\cdot \ln ^\frac{1}{d}(n)}\right|>2\right)\leq\rho_h ,\,\, \forall n>N_h,
\end{equation*}
or, equivalently, $h_n = O_P(n^{-\frac{1}{d}} \ln ^{\frac{1}{d}}(n))$. 
\end{proof}

\begin{lemma}
For a distribution $\mathbb{P}$ over $\Omega$ that satisfies Definition \ref{NicelyDef}, let $X_n=\{x_1,x_2,...,x_n\}$ be a set of $n$ $iid$ samples of $\mathbb{P}$. Denote by $\delta_n$ the separation of $X_n$. Then, 
         \begin{equation}
            \delta_n = \Omega_P(n^{-\frac{2}{d}}).
         \end{equation} 
\label{Order in probability of delta_n}
\end{lemma}

\begin{remark}
    A version of the bound in Lemma \ref{Order in probability of delta_n} was previously established by \cite{brauchart2018random} for the specific case of a sphere under a uniform distribution. Our result generalizes this bound to any compact domain in $\mathbb{R}^d$ and relaxes the distribution requirement to almost uniform distributions.
\end{remark}

\begin{proof}
    Recall that $P( B(x,\varepsilon))$ is the probability to have an  $\varepsilon$-ball centered at $x$ with a sample in it. We've shown  that this probability is bounded from above by $\bar{c}\cdot P_{U}(B(x,\varepsilon))$, and bounded from below by $\underline{c}\cdot P_{U}(B(x,\varepsilon))$. 
    Note that $P(\delta_n>\varepsilon)$ represents the probability that no two samples from $X$ occupy the same $\varepsilon$-ball around either of them. Thus, considering $\varepsilon$-balls around all samples and using the fact that these balls may not be disjoint, we get the following bound
\begin{equation}
      P(\delta_n>\varepsilon)\geq \prod_{i=1}^{n-1} (1-\sum_{j=1}^i P(B(x_j,\varepsilon)).
 \end{equation}
 Hence, 
\begin{equation}
      P(\delta_n<\varepsilon)\leq 1- \prod_{i=1}^{n-1} (1-\sum_{j=1}^i P(B(x_j,\varepsilon)).
 \end{equation}
 Using the upper bound of $P(B(x,\varepsilon))$, we obtain that 
 \begin{equation}
 1-\prod_{i=1}^{n-1} (1-\sum_{j=1}^i P(B(x_j,\varepsilon)) \leq 1-\prod_{i=1}^{n-1} (1-i \cdot \bar{c} \cdot P_{U}(B(x_i,\varepsilon))),
 \end{equation}
if $c'=\Bar{c}\cdot C$, then by previous results,
\begin{equation}
 1- \prod_{i=1}^{n-1} (1-i \cdot P\left(B(x_i,\varepsilon)\right))\leq 1- \prod_{i=1}^{n-1} (1-i \cdot c' \varepsilon^d).
\end{equation}
Hence, 
\begin{equation}
 P\left(\delta_n<\varepsilon\right) \leq 1-\prod_{i=1}^{n-1} (1-i \cdot c' \varepsilon^d).
\end{equation}
Now,
\begin{equation}
    1-\prod_{i=1}^{n-1}\left(1-i \cdot c' \varepsilon^d\right) \leq 1- \prod_{i=1}^{n-1}\left(1-nc' \varepsilon^d\right)=1-\left(1-nc' \varepsilon^d\right)^n,
\end{equation}
i.e., 
\begin{equation}
  P\left(\delta_n<\varepsilon\right) \leq 1- \left(1-nc' \varepsilon^d\right)^n.
  \label{delta<eps}
\end{equation}
Fix any $\rho_\delta \in (0,1)$. We seek to find a sequence $\varepsilon(n)$ such that the right-hand side of \eqref{delta<eps} is bounded by $\rho_\delta$. Let $\varepsilon(n)$ be defined by setting:
\begin{equation}
    1 - (1 - n c' \varepsilon^d)^n = \rho_\delta.
\end{equation}
Solving for $\varepsilon(n)$ yields:
\begin{equation}
    \varepsilon(n) = \left( \frac{1 - (1 - \rho_\delta)^{1/n}}{n c'} \right)^{1/d}.
    \label{eps_exact}
\end{equation}
To obtain a clean bound for \eqref{eps_exact}, we observe that $(1 - \rho_\delta)^{1/n} = \exp\left( \frac{1}{n} \log(1 - \rho_\delta) \right)$. Applying the inequality $e^{-x} \geq 1 - x$ for $x \geq 0$, where $x = \frac{1}{n} \log\left(\frac{1}{1 - \rho_\delta}\right)$, we have:
\begin{equation}
    (1 - \rho_\delta)^{1/n} \geq 1 - \frac{\log(1 / (1 - \rho_\delta))}{n}.
\end{equation}
Rearranging this inequality gives:
\begin{equation}
    1 - (1 - \rho_\delta)^{1/n} \leq \frac{\log(1 / (1 - \rho_\delta))}{n}.
\end{equation}
Substituting this into \eqref{eps_exact}, we define the constant $K_\delta$ as:
\begin{equation}
    K_\delta := \left( \frac{\log(1 / (1 - \rho_\delta))}{c'} \right)^{1/d}.
\end{equation}
It follows that for all $n \geq 1$:
\begin{equation}
    \varepsilon(n) \leq K_\delta \, n^{-2/d}.
\end{equation}
Consequently, substituting this bound back into \eqref{delta<eps}, we obtain:
\begin{equation}
    P\left( \delta_n < K_\delta \, n^{-2/d} \right) \leq \rho_\delta, \quad \forall n \geq 1.
\end{equation}
By the definition of order in probability for lower bounds, this strictly implies:
\begin{equation}
    \delta_n = \Omega_P(n^{-2/d}).
\end{equation}

\end{proof}

\begin{coro}
    Since $h_n = O_P(n^{-\frac{1}{d}} \ln ^{\frac{1}{d}}(n))$ and $\delta_n = \Omega_P(n^{-\frac{2}{d}})$ we get that the quasi-uniform property does not hold under the sampling assumptions of Definition \ref{h-delta-condition -levin}.
    \label{Failure of Quasi-corollary}
\end{coro}
Corollary \ref{Failure of Quasi-corollary} poses a challenge, as the quasi-uniform property is a common assumption, particularly in the deterministic MLS method. Consequently, the proofs in \cite{levin1998approximation}, \cite{mirzaei2015analysis}, and \cite{wendland2005scattered} cannot be directly applied to a stochastically generated data set, necessitating modifications to their arguments.
Our proposed solution is to prove Theorem \ref{main_result}, assuming that the stochastic data is well behaved distribution, as defined in Definition \ref{NicelyDef}.

As a preliminary step, we first establish a fundamental lemma that captures a key asymptotic property of the sample size as $n\rightarrow\infty$, which underpins each of these results.

\begin{lemma}
Let $D_n$ be a ball of radius $R_n \propto \left(\frac{\log n}{n}\right)^{1/d}$, centered at some point $x\in \Omega$ and let 
$$ 
N_{D_n} = \sum_{i=1}^n \mathbf{1}_{\{x_i \in D_n\}}$$
denote the number of sample points lying in $D_n$. Then, there exist finite constants $0 < \gamma_1 \leq \gamma_2$ such that for some constant $c>0$,
$$ P\left( \gamma_1 \log n \ \leq\ N_{D_n} \ \leq\ \gamma_2 \log n \right) \geq 1-\frac{2}{n^c},
$$
i.e.,
\[
\lim_{n \to \infty} P\left( \gamma_1 \log n \ \leq\ N_{D_n} \ \leq\ \gamma_2 \log n \right) = 1.
\]
    \label{convergence rate of sample size growth}
\end{lemma}

\begin{proof}
Let \( N_{D_n} = \sum_{i=1}^n \mathbf{1}_{\{x_i \in D_n\}} \). Each indicator variable \( \mathbf{1}_{\{x_i \in D_n\}} \) is Bernoulli with success probability \( p_n = P(x_i \in D_n) \). Then \( N_{D_n} \sim \text{Bin}(n, p_n) \), and:
\[
\mathbb{E} N_{D_n} = n p_n = n \int_{D_n} p(x) dx.
\]
Since \( D_n \) is a ball of radius \( R_n  \propto \left(\frac{\log n}{n}\right)^{1/d} \), its volume is
\[
\operatorname{Vol}(D_n) \propto  R_n^d \propto \frac{\log n}{n}.
\]
Because \( p(x) \) is bounded below by a constant \( \underline{c} > 0 \) on its support, we obtain:
\begin{equation}
    \mathbb{E} N_{D_n} \geq n \cdot \int_{D_n} \underline{c} \, dx = \underline{c} \cdot n \cdot \operatorname{Vol}(D_n) \geq c_1 \cdot \log n
    \label{lower bound inequality ENDn}
\end{equation}
for some constant \( c_1 > 0 \). Likewise, since \( p(x) \leq \bar{c} \), we also get:
\begin{equation}
    \mathbb{E} N_{D_n} \leq \bar{c} \cdot n \cdot \operatorname{Vol}(D_n) \leq c_2 \cdot \log n.
\label{upper bound inequality ENDn}
\end{equation}
We now apply the Chernoff bound: If \( X \sim \text{Bin}(n, p) \) with \( \mu = \mathbb{E}X \), then for any \( 0 < \delta < 1 \),
\[
P\left(|X - \mu| \geq \delta \mu\right) \leq 2 \exp\left(-\frac{\delta^2 \mu}{3}\right).
\]
Applying this to \( N_{D_n} \) with \( \mu = \mathbb{E} N_{D_n} \), we obtain:
\[
P\left( \left| N_{D_n} - \mathbb{E} N_{D_n} \right| \geq \delta \mathbb{E} N_{D_n} \right) \leq 2 \exp\left(-c \log n\right) = \frac{2}{n^c},
\]
where $c$ is a positive constant dependents on $\delta$.
Thus, 
\begin{equation}
    P\left((1 - \delta) \mathbb{E} N_{D_n}\leq N_{D_n} \leq (1 + \delta) \mathbb{E} N_{D_n}\right)\geq 1-\frac{2}{n^c}.
\end{equation}
Denote $\gamma_1=(1-\delta)c_1$ and  $\gamma_2=(1+\delta)c_2 $, and use the inequalities in \eqref{lower bound inequality ENDn},\eqref{upper bound inequality ENDn} we obtain that
\begin{equation}
    P\left( \gamma_1 \log n \ \leq\ N_{D_n} \ \leq\ \gamma_2 \log n \right) \geq 1-\frac{2}{n^c},
    \label{the probability of expected number of samples in D_n}
\end{equation}
which implies that 
$$
\lim_{n \to \infty} P\left( \gamma_1 \log n \ \leq\ N_{D_n} \ \leq\ \gamma_2 \log n \right) = 1.
$$
\end{proof}
We denote by $\mathscr{E}_I $ the event  
\begin{equation}
    \mathscr{E}_I:=\{ \# I_{B(\hat{x}, s h_n)} \in [\gamma_1 \log n,\gamma_2 \log n] \},
\label{E_I event}
\end{equation}
where $I_{B(\hat{x}, s h_n)}$ is the index set around a fixed but arbitrary point $\hat{x}\in\Omega$:
$$
I_{B\left(\hat{x},sh_n\right)}=\left\{i: x_i \in B\left(\hat{x},sh_n\right)\right\},
$$
and $B\left(\hat{x}, sh_n\right)=\left\{x \in \mathbb{R}^d:\|x-\hat{x}\| \leq sh_n\right\}$.

We wish to emphasize an important point regarding the result of Lemma \ref{convergence rate of sample size growth}. As the total number of samples increases, the number of samples contained in any fixed ball tends to infinity, providing the formal basis of the failure in Theorem \ref{theorem of order of probability}.  
In the stochastic setting, however, this phenomenon is not only unavoidable but also advantageous, since it is natural for the quality of the approximation to improve as the number of samples grows. By adapting the proof techniques to accommodate this behavior, together with our results, the difficulty is resolved, and the conclusions obtained in the deterministic framework can be retained.

\section{Proof of Theorem \ref{main_result}} \label{Main_Result_proof}
We first establish Lemma \ref{main_result_assisting}, which is a refined version of Theorem \ref{main_result}. In this lemma, the analysis is carried out for individual partial derivatives of fixed order, represented by multi-index derivatives. This allows us to treat each derivative separately and obtain precise bounds. We then derive Theorem \ref{main_result} as a direct consequence by applying the lemma to general linear differential operators with constant coefficients.

\begin{lemma}
    Let $X_n$ be a sample set of size $n$ (sufficiently large) drawn from a nicely behaving distribution \emph{(}Definition \ref{NicelyDef}\emph{)} on $\Omega$, which satisfies the interior cone condition \emph{(}Definition \ref{cone_cond}\emph{)}. 
    Let $F_n$ be the corresponding sampled values of some function $f:\mathbb{R}^d\rightarrow\mathbb{R}$, and let $\mlsfunstoch(x)$ be the value of the $(k-1)$ degree MLS function at a point $x$, then for any $\hat{x}\in \Omega$ and any multi-index $\alpha$ with  $|\alpha|\leq k-1$,

    \begin{equation}
        P\left(\left|\partial^{\alpha} \mlsfunstoch(\hat{x})-\partial^{\alpha} f(\hat{x})\right| \leq\mathcal{K}\cdot h_n^{k-|\alpha|}\right)\geq 1-\frac{1}{n^{c_\mu}},
    \end{equation}
    where $\mathcal{K}$ and $c_\mu$ are some positive constants.  \label{main_result_assisting}
\end{lemma}

\begin{coro}
    \begin{equation}
    \lim_{n\rightarrow\infty}P\left(\left|\partial^{\alpha} \mlsfunstoch(\hat{x})-\partial^{\alpha} f(\hat{x})\right| \leq\mathcal{K}\cdot h_n^{k-|\alpha|}\right)=1.
    \end{equation}  
\end{coro}

\begin{proof}

Let $X_n=\left\{x_i\right\}_{i=1}^n\subset\mathbb{R}^d$ be a set of scattered data points.

In the following, we work under the assumption that the event $\mathscr{E}_I$ defined in \eqref{E_I event} holds;  under this event, by \eqref{the probability of expected number of samples in D_n}, the expected number of samples in the index set
\( I_{B(\hat{x}, s h_n)} \) is contained in $[\gamma_1 \log n,\gamma_2 \log n]$ with probability at least \( 1 - \frac{2}{n^c} \). 

In \cite{levin1998approximation} it was shown that the MLS approximation is given by
$$
\mlsfunstoch(x)=\sum_{i \in I_{B\left(\hat{x}, s h_n\right)}} a_i^*(x) f\left(x_i\right),
$$
where the coefficients $a_i^*(x)$ are called \textit{shape functions}. These functions are obtained by minimizing a quadratic form, a procedure rooted in the foundational work of Backus and Gilbert \cite{backus1968resolving, backus1970uniqueness} regarding optimal local averaging kernels. Specifically, the shape functions are defined as:
$$
a_i^*(x)=\arg\min_{a_i(x)}\sum_{i \in I_{B\left(\hat{x}, s h_n\right)}} a_i(x)^2 \frac{1}{\theta_h\left(x_i-x\right)},
$$
under the constraints
$$
\sum_{i \in I_{B\left(\hat{x}, s h_n\right)}} a_i(x) p\left(x_i\right)=p(x), \quad \forall p \in \Pi_{k-1}\left(\mathbb{R}^d\right),
$$
and $\theta_h(x_i-x)$ is a weight function satisfies the properties in \eqref{theta_weight func}.
Let $A$ denote the number of multi-indices $\alpha = (\alpha_1, \dots, \alpha_d)$ with nonnegative entries such that $|\alpha|\leq k$. For a multi-index $\alpha$ and a vector $v \in \mathbb{R}^d$, we define $v^\alpha = \prod_{j=1}^d v_j^{\alpha_j}$, where $|\alpha| = \sum_{j=1}^d \alpha_j$ denotes the order of the multi-index.
It is shown in \cite{mirzaei2015analysis} that the shape functions can be expressed  in the following form:
\begin{equation}
    a_i^*(x)=\theta_h(x_i-x) \sum_{|\alpha| \leq k-1} \eta_\alpha(x) \frac{\left( x_i-x\right)^\alpha}{h_n^{|\alpha|}}, \quad i=1,2, \ldots, n.
    \label{explicitly a*_i}
\end{equation}
where $\eta_\alpha(x)$ are the components of $\eta(x)\in\mathbb{R}^A$. The polynomial--reproduction property of the shape functions determines $\eta(x)$ as the solution of the positive definite system
\begin{equation}
\mathscr{G}_n(x)\,\eta(x)=\boldsymbol{p}(x),
\label{gram system}
\end{equation}
where $\boldsymbol{p}(x)\in \mathbb{R}^A$ is the vector whose components are $p_\alpha(x)$,
\begin{equation}
    p_\alpha(x)=\left\{\frac{(x-\hat{x})^\alpha}{h_n^{|\alpha|}}\right\}_{0 \leq|\alpha| \leq k-1},
    \label{eq:polynomial basis}
\end{equation}
and $\mathscr{G}_n(x)\in \mathbb{R}^{A\times A}$ is the \emph{(unnormalized)} Gram matrix with entries
\begin{equation}
    \mathscr{G}_{n\alpha \beta}(x)=\sum_{i=1}^n \theta_h(x_i - x) \frac{(x_i - x)^\alpha}{h_n^{|\alpha|}} \cdot \frac{(x_i - x)^\beta}{h_n^{|\beta|}}, \quad \alpha, \beta =1, \ldots, A.
    \label{A_nab in main proof}
\end{equation}
The normalized moment matrix $\mathscr{A}_n(x)$ analyzed below---the empirical weighted covariance matrix, which converges to the theoretical weighted covariance matrix as $n\to\infty$---is obtained by averaging $\mathscr{G}_n$ over the active samples,
\begin{equation}
\mathscr{A}_n(x)=\frac{1}{N_{B(\hat{x}, s h_n)}}\,\mathscr{G}_n(x),
\qquad N_{B(\hat{x}, s h_n)}:=\#I_{B(\hat{x}, s h_n)} .
\label{gram normalized relation}
\end{equation}
Consequently the reproduction system \eqref{gram system} is equivalent to
\begin{equation}
\eta(x)=\mathscr{G}_n^{-1}(x)\,\boldsymbol{p}(x)
       =\frac{1}{N_{B(\hat{x}, s h_n)}}\,\mathscr{A}_n^{-1}(x)\,\boldsymbol{p}(x).
\label{eta normalized}
\end{equation}
To illustrate the approximation of a function by the stochastic MLS operator, we follow Levin's approach \cite{levin1998approximation}. Given an arbitrary $\hat{x}\in\Omega$, we define the MLS error as
$$
E_{F_n,X_n}(\hat{x})=\mlsfunstoch(\hat{x})-f(\hat{x}).
$$
By adding and subtracting a polynomial $p(x)$ of degree $\leq k$ and  centered around $\hat{x}$, we get
\begin{equation}
    E_{F_n,X_n}(\hat{x})=\sum_{i \in I_{B\left(\hat{x}, s h_n\right)}} a_i^*(\hat{x})\left(f\left(x_i\right)-p\left(x_i\right)\right)+p(\hat{x})-f(\hat{x}).
    \label{E_(F_n,X_n)}
\end{equation}
Denote by $p_f$ the taylor polynomial of $f(x)$ around $\hat{x}$. We have:
\begin{equation}
    \max _{x \in B(\hat{x}, s h_n)}|f(x)-p_f(x)| \leq C|f|_{\mathcal{C}^k\left(B\left(\hat{x},s h_n\right)\right)} h_n^k,
    \label{taylor reminder}
\end{equation}
where the seminorm, $|f|_{\mathcal{C}^k\left(B\left(\hat{x},s h_n\right)\right)} $, is defined by
$$
|f|_{\mathcal{C}^k\left(B\left(\hat{x},s h_n\right)\right)} \triangleq \max _{\substack{|\alpha|=k \\ x \in B\left(\hat{x}, s h_n\right)}}\left|\partial^\alpha f(x)\right|
,$$
with $\alpha=\left(\alpha_1, \ldots, \alpha_d\right)$ denoting a multi-index of order $|\alpha|=k$, and
$$
\partial^\alpha f(x)=\partial_{x_1}^{\alpha_1} \partial_{x_2}^{\alpha_2} \cdots \partial_{x_d}^{\alpha_d} f(x).
$$
Using $p_f$ as the polynomial in Equation \eqref{E_(F_n,X_n)} we get
$$
\left|E_{F_n,X_n}(\hat{x})\right| \leq C|f|_{\mathcal{C}^k\left(B\left(\hat{x},s h_n\right)\right)} h_n^k .
$$
Since the focus is on measuring the error in the derivatives, we obtain
$$
\partial^\alpha E_{F_n,X_n}(x)=\sum_{i \in I_{B\left(x, h_n\right)}}\left(\partial^\alpha a_i^*(x)\right)\left(f\left(x_i\right)-p_f\left(x_i\right)\right)+\partial^\alpha(p_f(x)-f(x)).
$$
Note that in the current case, $\partial^\alpha$ is scalar-valued as $f: \mathbb{R}^d \rightarrow \mathbb{R}$. Then, for a fixed but arbitrary $\hat{x} \in \Omega$ we get
\begin{equation}
  \left|\partial^\alpha E_{F_n,X_n}(\hat{x})\right| \leq \max _{x \in B(\hat{x},s h_n)}|f(x)-p_f(x)| \sum_{i \in I_{B\left(\hat{x}, s h_n\right)}}\left|\partial^\alpha a_i^*(\hat{x})\right|+\left|\partial^\alpha p_f(x)-\partial^\alpha f(x)\right|.  
  \label{bound_on_dE}
\end{equation}
Since $p_f$ is the Taylor expansion of $f(x)$ at $\hat{x}$ and $\|\hat{x}-x\| \leq sh_n$ we get that
\begin{equation}
    \left|\partial^\alpha p_f(x)-\partial^\alpha f(x)\right| \leq c_1|f|_{\mathcal{C}^k(B(\hat{x},s h_n))}h_n^{k-|\alpha|}.
     \label{bound_on_dp-df}
\end{equation}
Given the event $\mathscr{E}_I$ hold, we show in Lemma \ref{lemma 3.11 davoud-bound shape function} that for all $x \in B(\hat{x},s h_n)$, 
\begin{equation}
    P\left(\sum_{i \in I_{B\left(x, s h_n\right)}}\left|\partial^\alpha a_i^*(x)\right| \leq C_\alpha^{\prime} \cdot h_n^{-|\alpha|}\Big|\mathscr{E}_I\right)\geq 1-\frac{1}{n^{s_1}},
    \label{probability _bound_on_da*}
\end{equation}
where $C'_{\alpha}$ is a positive constant.
Therefor, using \eqref{probability _bound_on_da*} to bound the probability of the event $\mathscr{E}_I$, we have, 

\begin{equation}
    P\left(\sum_{i \in I_{B\left(x, s h_n\right)}}\left|\partial^\alpha a_i^*(x)\right| \leq C_\alpha^{\prime} \cdot h_n^{-|\alpha|}\right)\geq 1-\frac{2}{n^c}-\frac{1}{n^{s_1}}.
    \label{bound on da* with E_I}
\end{equation}
Hence, exists a constant  $c_{\mu}$ such that
\begin{equation}
    P\left(\sum_{i \in I_{B\left(x, s h_n\right)}}\left|\partial^\alpha a_i^*(x)\right| \leq C_\alpha^{\prime}\cdot h_n^{-|\alpha|}\right)\geq 1-\frac{1}{n^{c_\mu}}.
\end{equation}
Substituting the bounds from Equations \eqref{bound_on_dp-df} and \eqref{bound on da* with E_I} into Equation \eqref{bound_on_dE}, and using the fact that $\max_{x \in B(\hat{x},h_n)} |f(x) - p_f(x)|$ is bounded by $h_n^k$ up to a positive constant (Taylor remainder estimate, \eqref{taylor reminder}), we evaluate at $\hat{x}$, and we obtain that exits a positive constant $\mathcal{K}$ such that
$$ 
P\left(\left|\partial^{\alpha} \mlsfunstoch(\hat{x})-\partial^{\alpha} f(\hat{x})\right| \leq\mathcal{K}\cdot h_n^{k-|\alpha|}\right)\geq 1-\frac{1}{n^{c_\mu}}.
$$
That is,
$$ \lim_{n\rightarrow\infty}P\left(\left|\partial^{\alpha} \mlsfunstoch(\hat{x})-\partial^{\alpha} f(\hat{x})\right| \leq\mathcal{K}\cdot h_n^{k-|\alpha|}\right)=1,
$$
as desired. 
\end{proof}

\begin{proof}[Proof of \emph{\textit{Theorem} \ref{main_result}}]
Let $Q$ be a linear differential operator of order $m$, which can be written as
\[
Q=\sum_{|\alpha|\le m} q_\alpha\,\partial^\alpha,
\]
where the coefficients $q_\alpha$ are real constants. Let
\[
M_Q:=\{\alpha:\,|\alpha|\le m,\ q_\alpha\neq 0\},
\]
which is a finite index set.
Define the vector $\mathscr{Q}_n\in\mathbb{R}^{\#(M_Q)}$ by
\[
(\mathscr{Q}_n)_\alpha=\frac{\alpha!\,q_\alpha}{h_n^{|\alpha|}},
\qquad \alpha\in M_Q.
\]
With this notation, the operator $Q$ can be expressed as
\[
Q=\mathscr{Q}_n^{\prime}\,\partial^\alpha.
\]
By Lemma \ref{main_result_assisting}, for each fixed $\alpha\in M_Q$,
\[
P\!\left(\left|\partial^\alpha \mlsfunstoch(\hat{x})
-\partial^\alpha f(\hat{x})\right|
\ge \mathcal{K}\,h_n^{k-|\alpha|}\right) \leq \frac{1}{n^{c_\mu}}.
\]
Applying the union bound over the finite set $M_Q$, we obtain
\begin{align*}
&P\!\left(
\max_{\alpha\in M_Q}
\left|\partial^\alpha \mlsfunstoch(\hat{x})
-\partial^\alpha f(\hat{x})\right|
\ge \mathcal{K}\,h_n^{k-|\alpha|}
\right) \\
&\qquad\le
\sum_{\alpha\in M_Q}
P\!\left(\left|\partial^\alpha \mlsfunstoch(\hat{x})
-\partial^\alpha f(\hat{x})\right|
\ge \mathcal{K}\,h_n^{k-|\alpha|}\right)
\leq \#(M_Q)\cdot\frac{1}{n^{c_\mu}} =\frac{C_{M_Q}}{n^{c_\mu}},
\end{align*}
where $C_{M_Q}$ is a constant independent of $n$.
Consequently,
$$
P\!\left(\left|Q\mlsfunstoch(\hat{x})-Qf(\hat{x})\right|
\ge \mathcal{K}\,h_n^{k-m}\right)\leq \frac{C_{M_Q}}{n^{c_\mu}},
$$
i.e.,
\[
P\!\left(\left|Q\mlsfunstoch(\hat{x})-Qf(\hat{x})\right|
\ge \mathcal{K}\,h_n^{k-m}\right)
\;\xrightarrow[n\to\infty]{}\;0,
\]
which proves the theorem.
\end{proof}

In the proof of Lemma \ref{main_result_assisting}, we rely on a series of auxiliary lemmas. These lemmas provide probabilistic bounds for key components of the MLS method, including the weight function, the matrix $\mathscr{A}_n(x)$, its inverse, and the shape functions. The results are used to control the behavior of the approximation and establish the required convergence properties.

We begin by controlling the derivatives of the matrix $\mathscr{A}_n(x)$ \eqref{gram normalized relation}, which appears (up to normalization) in the normal equations of the MLS method.

\begin{remark}
    For simplicity of the writing, in Lemmas (\ref{lemma 3.6 davoud -bound on dA}), (\ref{lemma 3.7 davoud - lambda_min}), and (\ref{lemma 3.10 davoud - bound a^-1}), 
the point \( x \in \Omega \) is assumed to be an interior point of \( \Omega \). 
For points on the boundary, the results remain valid, due to the interior cone condition, although the constants may differ. 
By taking the smaller value for the lower bound and the larger value for the upper bound, 
the statements can be extended to cover all cases.
\end{remark}

\begin{lemma}
    Let the event $\mathscr{E}_I$ defined in \eqref{E_I event} hold. Then, for a fixed but arbitrary evaluation point ${x}\in \Omega$,
    \begin{equation}
    P\left(\|\partial^{\alpha} \mathscr{A}_n({x})\| \leq \tilde{C}_\alpha h_n^{-|\alpha|}\Big|\mathscr{E}_I\right) \geq 1 - \frac{2}{n^{\tilde{c}}}, 
\end{equation}
    where $\tilde{C_\alpha},\tilde{c}$ are positive constant, and $\mathscr{A}_n$ is the matrix defined in \eqref{gram normalized relation}.
    \label{lemma 3.6 davoud -bound on dA}
\end{lemma}

\begin{proof}

Let $B\left(\hat{x}, h_n\right)=\left\{x \in \mathbb{R}^d:\|x-\hat{x}\| \leq h_n\right\}$. 
Denote $I_{B\left(\hat{x}, h_n\right)}=\left\{i:X_i\in B(\hat x, h_n)\right\}$ and $N_{B\left(\hat{x}, h_n\right)}=\#I_{B\left(\hat{x}, h_n\right)}$ is the number of elements in this ball.
For $\Pi_{k-1}^d$ we consider the basis
\begin{equation}
    p_\alpha(x)=\left\{\frac{(x - \hat x)^{\alpha}}{h_n^{|{\alpha}|}}\right\}_{0 \leqslant|{\alpha}| \leq k-1},
    \label{poly_basis}
\end{equation}
which is the monomial bases centered around the point $\hat x$ and scaled by the fill distance.
This choice of scaled monomial basis is made for convenience. The MLS formulation, and in particular the matrix $\mathscr{A}_n$ and its derivatives, are invariant under a change of basis in $\Pi_{k-1}^d$. Consequently, the estimates derived below do not depend on the specific choice of polynomial basis and hold for any basis of $\Pi_{k-1}^d$.
Thus, the MLS function can be written as
\begin{equation}
    \mlsfunstoch(x)=\sum_{\alpha\in A} \hat{b}_{n \alpha}(x) p_\alpha(x)
    \label{mls_as_sum}
,\end{equation}
where $A$ denotes the collection of all $d$-tuples $\alpha$ of nonnegative integers such that $|\alpha| \leq k-1$. The expression $\hat{b}_{n\alpha}(x)$ is the $\alpha$-coordinate of the vector of size $A$ which is given by
\begin{equation}
    \hat{b}_n(x)=\mathscr{A}_n^{-1}(x) \mathscr{X}_n^{\top}(x) F_n,
    \label{b_n-alpha}
\end{equation}
where $\mathscr{X}_n(x)$ and $\mathscr{A}_n(x)$ are defined as follows:
Using the basis defined in Equation (\ref{poly_basis}), we can construct the normal equations accordingly (the full formulation of the least-squares problem and the solution of the resulting normal equations are given in the \ref{Normal Equations Appendix}). The $n\times A$ design
matrix, $\tilde{\mathscr{X}}_n(x)$, takes the form 

\begin{equation*}
    \tilde{\mathscr{X}}_n(x) = \left(
\begin{array}{cccc}
\sqrt{\theta_1} p_{\alpha_0}(X_1) & \sqrt{\theta_1} p_{\alpha_1}(X_1) & \cdots & \sqrt{\theta_1} p_{\alpha_{|A|}}(X_1) \\
\vdots & & & \vdots \\
\sqrt{\theta_n} p_{\alpha_0}(X_n) & \sqrt{\theta_n} p_{\alpha_1}(X_n) & \cdots & \sqrt{\theta_n} p_{\alpha_{|A|}}(X_n)
\end{array}
\right),
\end{equation*}
where $\theta_i=\theta_h(x_i-x)$ and $p_{\alpha_j}; \,\,0\leq j\leq A$ is according to (\ref{poly_basis}).
From this, the normal equations are
\[
 \tilde{\mathscr{X}}_n(x)^\top \tilde{\mathscr{X}}_n(x) \hat{b}_n(x) =   \tilde{\mathscr{X}}_n(x)^\top F_n(x),
\]
which leads to the solution
\[
\hat{b}_n(x) = (\tilde{\mathscr{X}}_n(x)^\top \tilde{\mathscr{X}}_n(x))^{-1} \tilde{\mathscr{X}}_n(x)^\top F_n(x).
\]
Next, we normalize the design matrix by rescaling it: 
\[
\mathscr{X}_n(x) = \frac{1}{N_{B\left(x, h_n\right)}} \tilde{\mathscr{X}}_n(x),
\]
i.e.,
$$\mathscr{X}_{n i \alpha}(x)= \begin{cases}\frac{\sqrt{\theta_h(x_i-x)}}{N_{B\left(x, h_n\right)}} \frac{\left(x_i-x\right)^\alpha}{h_n^{|\alpha|}}, & i \in I_{B\left(x, h_n\right)} \text { and } \alpha \in A, \\ 0, & i \notin I_{B\left(x, h_n\right)} \text { and } \alpha \in A.\end{cases}$$
Note, as stated in the Lemma, $x$ is fixed, hence, once we fix $x$, the magnitude $ N_{B\left(x, h_n\right)}$ depends only on $n$.
Using this notation, let $\mathscr{A}_n(x)$ be the $A \times A$ matrix defined by 
\begin{align*}
\mathscr{A}_n(x) 
&= \frac{1}{N_{B\left(x, h_n\right)}} \cdot \tilde{\mathscr{X}}_n(x)^\top \tilde{\mathscr{X}}_n(x) \\
&= \frac{1}{N_{B\left(x, h_n\right)}} \cdot N_{B\left(x, h_n\right)}^2 \cdot \mathscr{X}_n(x)^\top \mathscr{X}_n(x) \\
&= N_{B\left(x, h_n\right)} \cdot \mathscr{X}_n(x)^\top \mathscr{X}_n(x).
\end{align*}
Therefore, each matrix element is given by
\begin{equation}
    \mathscr{A}_{n \alpha \beta}(x) =\frac{1}{N_{B\left(x, h_n\right)}} \sum_{I_{B\left(x, h_n\right)}} \theta_h(x_i-x)\frac{\left(x_i-x\right)^\alpha\left(x_i-x\right)^\beta}{h_n^{|\alpha|+|\beta|}}. 
    \label{Anab}
\end{equation}
With these notations, we obtain the terms $\mathscr{X}_n(x)$, $\mathscr{A}_n(x)$  and $\hat{b}_{n }(x)$.
Now, note that $\mathscr{A}_{n}$ represents the empirical weighted covariance matrix of the random variable (i.e., the sample points $x_1,...,x_n$) expressed in the polynomial basis.
Define the normalized weight function by
$$
 w(x_i-x)=\frac{\theta_h(x_i - x)}{1/h_n\int_{B(0, sh_n)} \theta_h(u-x) d u} = \frac{\Phi(x_i/h_n - x/h_n)}{1/h_n\int_{B(0, sh_n)} \Phi(u/h_n-x/h_n) d u},
$$
using change of variables with $\left[t = \frac{u-x}{h_n}, u - x = t h_n, du = h_n dt\right]$ we obtain, 
$$
w(x_i-x)=\frac{\Phi(t_i)}{ 1/h_n\int_{t\in B(0, s)}\Phi(t) h_n dt} = \frac{\Phi(t_i)}{\int_{t\in B(0, s)}\Phi(t) dt}.
$$
where $t_i = \frac{x_i - x}{h_n}$.
Therefore,
\begin{equation}
    w(x_i-x)=\frac{\Phi(t_i)}{\int_{t\in B(0, s)}\Phi(t) dt},
    \label{definition of w(xi-x)}
\end{equation}
so that $w(x)$ is a probabilty density function over the ball $B(0,sh)$.
Hence, the normalized sample covariance matrix is given by
$$\mathscr{A}_{n \alpha \beta}^w(x)=N_{B\left(x, h_n\right)} \sum _{I_{B\left(x, h_n\right)}} w(x_i-x)\mathscr{X}_{n i \alpha}(x) \mathscr{X}_{n i\beta}(x) .$$
Hence,
\begin{equation}
    \mathscr{A}^w_n = \mathscr{A}_n \cdot \frac{1}{\int_{t\in B(0, s)}\Phi(t) dt}.
    \label{const of Anw_An}
\end{equation}
Define the $A\times A$  matrix $\mathscr{A}^w$ by it's entries 
$$
\mathscr{A}^w_{\alpha \beta}={\int_{|z-x| \leq s}(z-x)^\alpha(z-x)^\beta w(z-x) d z}.
$$
The matrix $\mathscr{A}^w$ is symmetric and positive definite; consequently, its determinant is strictly positive, $\operatorname{det} \mathscr{A}^w>0$. Conceptually, $\mathscr{A}^w$ is a weighted covariance matrix, and the matrices $\mathscr{A}(x)$ and $\mathscr{A}^w_n(x)$ both serve as covariance matrices of the basis functions, with $\mathscr{A}^w_n(x)$ being a scaled version of $\mathscr{A}(x)$ that approximates the theoretical weighted covariance matrix $\mathscr{A}^w$. The primary difference lies in the normalization factors, $N_{B\left(x, h_n\right)}$ and $\frac{1}{\int_{t\in B(0, s)}\Phi(t - t_i) dt} $, making them equivalent up to a multiplicative constant.

By the Scalar Bernstein inequality, for every $\varepsilon>0$,
\begin{equation}
    P\left(\left|\mathscr{A}_{n a \beta}^w(x)-\mathscr{A}^w_{\alpha \beta}\right| \leq \varepsilon \Big|  \mathscr{E}_I\right)\geq 1-2 e^{-c'\cdot \varepsilon^2 \log(n) }=1-\frac{2}{n^{c'\cdot \varepsilon^2 }},
    \label{eq:bernstein inequality}
\end{equation}
where $c'$ is a positive constant. So if we denote $C^{\Phi}=\int_{t\in B(0, s)}\Phi(t) dt$, then by Equation (\ref{const of Anw_An})
\begin{equation}
  P\left(\left|\mathscr{A}_{n a \beta}(x)-C^{\Phi}\cdot\mathscr{A}^w_{\alpha \beta}\right| \leq \varepsilon \Big| \mathscr{E}_I\right)\geq 1-\frac{2}{n^{c'\cdot \varepsilon^2 }}.
    \label{convergence A to Aw}
\end{equation}
Now, by Lemma \ref{bound of weight fun} we obtain that for all $|\alpha| \leq k$, there exists a constant  $C_\alpha$ such that 

\begin{equation}
    \left|\partial^\alpha w(x_i-x)\right| \leq C_\alpha h_n^{-|\alpha|}, \quad \forall x \in \Omega.
    \label{bounded w(x)}
\end{equation}
Since $\Phi(x)$ is a compactly supported and $\mathcal{C}^k$ function, all of its derivatives up to order $k$ are continuous and bounded (see Lemma \ref{bound of weight fun}).
As only $w$ carries the $x$-dependence, the bound in \eqref{eq:bernstein inequality} applies verbatim to $\partial^\alpha \mathscr{A}_n$, whose mean is controlled by \eqref{bounded w(x)}.

Combining Equations (\ref{convergence A to Aw}) and (\ref{bounded w(x)}) together implies that 
\begin{equation}
     P(\|\partial^{\alpha} \mathscr{A}_n({x})\| - C_\alpha h_n^{-|\alpha|}\leq \varepsilon \Big|\mathscr{E} _I)\geq 1-\frac{2}{n^{c' \varepsilon^2 }}.
    \label{eq. 3.6 of Davoud}
\end{equation}
By choosing $\varepsilon$ such that
$c'\varepsilon^2=\tilde{c}$,  
we can simplify \eqref{eq. 3.6 of Davoud} to: 
\begin{equation} P\left(\|\partial^{\alpha} \mathscr{A}_n({x})\| \leq \tilde{C}_\alpha h_n^{-|\alpha|}\Big|\mathscr{E}_I\right) \geq 1 - \frac{2}{n^{\tilde{c}}}, 
\end{equation}
where $\tilde{C}_\alpha = C_\alpha + \varepsilon$ and $\tilde{c}$ is a positive constant.
\end{proof} 

Next, we bound the derivatives of the weight function used in the MLS formulation. This is necessary for controlling the growth of the entries in $\mathscr{A}_n(x)$ and related expressions.

\begin{lemma}
For a weight function $w(x_i-x)$, as defined in \eqref{definition of w(xi-x)}, we have
    $$\left|\partial^\alpha w(x_i-x)\right| \leq C_\alpha h_n^{-|\alpha|}, \quad \forall x \in \Omega.$$
    \label{bound of weight fun}
\end{lemma}

\begin{proof}
Differentiate \eqref{definition of w(xi-x)} with respect to $x$ 
$$  
\partial^\alpha w(x_i-x)=\partial^\alpha\left(\frac{\Phi(x_i/h_n - x/h_n)}{\int_{t\in B(0, s)}\Phi(x_i/h_n - u/h_n) du}\right).
$$
Since the denominator is independent of $x$, differentiation only affects the numerator. Using the chain rule,
$$
\partial^\alpha \Phi\left(\frac{x_i}{h_n}-\frac{x}{h_n}\right)=(-1)^{|\alpha|} h_n^{-|\alpha|}\left(\partial^\alpha \Phi\right)\left(\frac{x_i}{h_n}- \frac{x}{h_n}\right).
$$
Thus,
$$
\partial^\alpha w(x_i-x)=\frac{(-1)^{|\alpha|} h_n^{-|\alpha|}\left(\partial^\alpha \Phi\right)(x_i/h_n - x/h_n)}{\int_{t\in B(0, s)}\Phi(x_i/h_n - u/h_n) du}.
$$
Since $\Phi$ and its derivatives are bounded, there exists a constant $C_\alpha^{\prime}$ such that
$$
\left|\left(\partial^\alpha \Phi\right)(x_i/h_n - x/h_n)\right| \leq C_\alpha^{\prime}.
$$
Therefore,
$$
\left|\partial^\alpha w(x_i-x)\right| \leq \frac{C_\alpha^{\prime} h_n^{-|\alpha|}}{\int_{t\in B(0, s)}\Phi(x_i/h_n - u/h_n) du}.
$$
The denominator is independent of $x$ and is just a normalization factor ensuring $w(x_i-x)$ integrates to 1 . Then if we denote $C_\alpha=\frac{C'_\alpha}{\int_{t\in B(0, s)}\Phi(x_i/h_n - u/h_n) du}$ we obtain
$$
\left|\partial^\alpha w(x_i-x)\right| \leq C_\alpha h_n^{-|\alpha|}.
$$

\end{proof}

To ensure the stability of the MLS solution, we need a lower bound on the smallest eigenvalue of $\mathscr{A}_n(x)$, which guarantees invertibility with high probability.

\begin{lemma} 
Given the event $\mathscr{E}_I$ defined in \eqref{E_I event} holds, there exists a constant $C_\lambda>0$,  such that for a fixed but arbitrary evaluation point $x\in \Omega$,
$$
P\left(\lambda_{\min }\left(\mathscr{A}_n(x)\right) \geq C_\lambda \Big| \mathscr{E}_I\right)\geq 1-\frac{2A^2}{n^\phi},
$$
where $\phi$ is a real positive constant, $A$ is the dimension of the matrix $\mathscr{A}_n(x)$ and $\lambda_{\text {min }}\left(\mathscr{A}_n(x)\right)$ is the smallest eigenvalue of $\mathscr{A}_n(x)$. In particular, with the stated probability $\mathscr{A}_n(x)$ is positive definite and full rank.
    \label{lemma 3.7 davoud - lambda_min}
\end{lemma}

\begin{proof} 
Let $\pi=\left(\pi_\alpha\right)_{\alpha \in A}$ be a unit vector in $\mathbb{R}^{|A|}$ such that $\sum_{\alpha\in A} \pi_\alpha^2=1$. For $z\in\mathbb{R}^d$, we define the polynomial:
$$
P(z)=\sum_{\alpha \in A} \pi_\alpha z^\alpha
$$
Now define the empirical quadratic form:
$$
\mathcal{Q}_n(\pi):=\pi^{\top} \mathscr{A}_n(x) \pi.
$$
From \eqref{Anab} we get
$$
\mathcal{Q}_n(\pi)=\frac{1}{N_{B\left(x, h_n\right)}} \sum_{i \in I_{B\left(x, h_n\right)}} \theta_h\left(x_i-x\right) \cdot\left(\sum_\alpha \pi_\alpha \frac{\left(x_i-x\right)^\alpha}{h_n^{|\alpha|}}\right)^2 .
$$
Let us define the rescaled coordinates:
$$
t_i:=\frac{x_i-x}{h_n}, \quad \text { so that } \quad x_i=x+h_n t_i.
$$
Substituting $t_i=\frac{x_i-x}{h_n}$, and using $\theta_h\left(x_i-x\right)=\Phi\left(t_i\right)$, we get:

$$
\mathcal{Q}_n(\pi)=\frac{1}{N_{B\left(x, h_n\right)}} \sum_{i \in I_{B\left(x, h_n\right)}} \Phi\left(t_i\right) \cdot\left(\sum_\alpha \pi_\alpha t_i^\alpha\right)^2=\frac{1}{N_{B\left(x, h_n\right)}} \sum_{i \in I_{B\left(x, h_n\right)}} \Phi\left(t_i\right) \cdot P\left(t_i\right)^2,
$$
where $P(t)=\sum_\alpha \pi_\alpha t^\alpha$.
Given the event $\mathscr{E}_I$, the number of samples $\left\{t_i\right\}_{i \in I_{B\left(x, h_n\right)}}$ inside the rescaled ball $B(0, s)$ goes to infinity  as $n \rightarrow \infty$ in probability.  We now show that 
$\mathbb{E}\left[\mathcal{Q}_n(\pi)\right]$ is positive.
$$
\mathbb{E}\left(\mathcal{Q}_n(\pi)\right)=\frac{\int_{B\left(x, h_n\right)} \Phi\left(\frac{u-x}{h_n}\right) P\left(\frac{u-x}{h_n}\right)^2 p(u) d u}{\int_{B\left(x, h_n\right)} p(u) d u}.
$$
By applying the change of variables $t=\frac{u-x}{h_n}$, which means $d u=h_n^d d t$, we obtain:
$$
\mathbb{E}(\mathcal{Q}_n(\pi))=\frac{\int_{B(0, s)} \Phi(t) P(t)^2 p\left(x+h_n t\right) h_n^d d t}{\int_{B(0, s)} p\left(x+h_n t\right) h_n^d d t}.
$$
Recall now that $\int_{B(0,s)} \underline{c} d \mu \leq \int_{B(0,s)} d p \leq\int_{B(0,s)} \bar{c} d \mu$, where $p$ is a density function satisfies Definition \ref{NicelyDef}, hence,
\begin{equation}
\begin{aligned}
    \underline{\beta} \int_{B(0, s)} \Phi(t) P(t)^2 dt 
    &= \frac{\underline{c}\int_{B(0, s)} \Phi(t) P(t)^2 d t}{\bar{c}\int_{B(0, s)} d t} 
    \leq \mathbb{E}(\mathcal{Q}_n(\pi)) \\
    &\leq \frac{\bar{c}\int_{B(0, s)} \Phi(t) P(t)^2 d t}{\underline{c}\int_{B(0, s)} d t} 
    = \bar{\beta} \int_{B(0, s)} \Phi(t) P(t)^2 dt.
\end{aligned}
\end{equation}
where, $\underline{\beta}= \frac{\underline{c}}{\bar{c}\cdot \text{Vol}B(0,s)} $ and $\bar{\beta}=\frac{\bar{c}}{\underline{c}\text{Vol}B(0,s)} $. Denote,
$$
\mathcal{Q}(\pi)=\int_{B(0, s)} \Phi(t) \cdot P(t)^2 dt.
$$
Thus,
\begin{equation}
    \underline{\beta} \mathcal{Q}(\pi) \leq \mathbb{E}(\mathcal{Q}_n(\pi))\leq
\bar{\beta} \mathcal{Q}(\pi).
\label{bound of EQ_n}
\end{equation}
Since the integrand $\Phi(t) \cdot P(t)^2 $ is non-negative, and $P(t) $ is not identically zero on $ B(0,s)$, the integral $\mathcal{Q}(\pi) $ is strictly positive for all $\|\pi\|=1 $, i.e., $\mathcal{Q}(\pi)>0$.
Moreover, $\mathcal{Q}(\pi)$ is continuous in $\pi$, and is defined on a compact support, therefore,
by the Weierstrass Theorem $\mathcal{Q}(\pi)$ is bounded and attains its minimum.
Denote,
\begin{equation}
    \tau = \inf_{\|\pi\|=1} \mathcal{Q}(\pi).
    \label{pi*}
\end{equation}
In particular, for every $\pi$, $\mathcal{Q}(\pi)>0$ so we also have $\tau>0$. We fix from here on any constant
$$
C_\lambda \in \left(0,\ \underline{\beta}\,\tau\right).
$$
Consider the deterministic (conditional) covariance matrix
\begin{equation}
\mathscr{A}(x):=\mathbb{E}\!\left[\mathscr{A}_n(x)\,\middle|\,\mathscr{E}_I\right],
\end{equation}
so that for every unit vector $\pi$,
$$
\pi^\top\mathscr{A}(x)\pi=\mathbb{E}\!\left[\mathcal{Q}_n(\pi)\,\right]
\;\ge\;\underline{\beta}\,\mathcal{Q}(\pi)\;\ge\;\underline{\beta}\,\tau,
$$
where the first inequality is \eqref{bound of EQ_n}. Because $\mathscr{A}(x)$ is \emph{deterministic}, taking the infimum over $\|\pi\|=1$ is legitimate and yields
\begin{equation}
\lambda_{\min}\!\big(\mathscr{A}(x)\big)
=\inf_{\|\pi\|=1}\pi^\top\mathscr{A}(x)\pi
\;\ge\;\underline{\beta}\,\tau\;>\;C_\lambda\;>\;0 .
\label{eq:mean-pd}
\end{equation}
Similarly to Equation \eqref{eq:bernstein inequality}, by the scalar Bernstein inequality, for every $\varepsilon>0$ and every $\alpha,\beta$,
\begin{equation}
P\!\left(\big|\mathscr{A}_{n\alpha\beta}(x)-\mathscr{A}_{\alpha\beta}(x)\big|\ge\varepsilon
\Big|\mathscr{E}_I\right)\le 2e^{-c'\varepsilon^2\log n}=\frac{2}{n^{\,c'\varepsilon^2}}.
\label{eq:entry-conc}
\end{equation}
For an $A\times A$ matrix $B$, $\|B\|_{\mathrm{op}}\le\|B\|_F\le A\max_{\alpha,\beta}|B_{\alpha\beta}|$.
A union bound over the (finitely many, $n$-independent) $A^2$ entries therefore gives, from
\eqref{eq:entry-conc},
\begin{equation}
    P\!\left(\big\|\mathscr{A}_n(x)-\mathscr{A}(x)\big\|_{\mathrm{op}}\le A\varepsilon
\,\middle|\,\mathscr{E}_I\right)\ge 1-\frac{2A^2}{n^{\,c'\varepsilon^2}} .
\end{equation}
Note that $c'$ is the same constant as in \eqref{eq:bernstein inequality}: the constant is depending only on $\Phi$ and the support radius $s$, and the conditioning on $\mathscr{E}_I$ fixes the number of active samples $N_{B(\hat{x},sh_n)}=\Theta(\log n)$, so the Bernstein--Hoeffding exponent has the same dependence $c'\varepsilon^2\log n$ in both lemmas.

Choose $\varepsilon=\dfrac{\underline{\beta}\,\tau-C_\lambda}{A}$ (so $A\varepsilon=\underline{\beta}\,\tau-C_\lambda>0$) and set $\phi:=c'\varepsilon^2=\dfrac{c'\left(\underline{\beta}\,\tau-C_\lambda\right)^2}{A^2}>0$.
By Weyl's eigenvalue perturbation inequality for Hermitian matrices and \eqref{eq:mean-pd}, we obtain 
\begin{equation}
    \lambda_{\min}(\mathscr{A}_n(x))\;\ge\;\lambda_{\min}\!\big(\mathscr{A}(x)\big)
-\big\|\mathscr{A}_n(x)-\mathscr{A}(x)\big\|_{\mathrm{op}}
\;\ge\;\underline{\beta}\,\tau-\left(\underline{\beta}\,\tau-C_\lambda\right)\;=\;C_\lambda .
\end{equation}
Hence
$$
P\!\left(\lambda_{\min}(\mathscr{A}_n(x))\ge C_\lambda \Big|\mathscr{E}_I\right)
\;\ge\;1-\frac{2A^2}{n^{\phi}} ,
$$
which is the assertion (with the $n$-independent constant $2A^2$ in place of $2$). In
particular $\mathscr{A}_n(x)$ is positive definite and full rank for all sufficiently large $n$.
\end{proof}

Building upon the invertibility of $\mathscr{A}_n(x)$	and the bounds on its derivatives (established in Lemmas \ref{lemma 3.7 davoud - lambda_min} and \ref{lemma 3.6 davoud -bound on dA}, respectively), we now derive corresponding bounds for the derivatives of its inverse.
\begin{lemma}
    Given the event $\mathscr{E}_I$ defined in \eqref{E_I event} holds. Under the assumptions of Lemma \ref{lemma 3.6 davoud -bound on dA} and Lemma \ref{lemma 3.7 davoud - lambda_min}, for a fixed but arbitrary evaluation point $x \in \Omega \,\,\text{and}\,\,\, \alpha, \beta \in A$, we have:
    \begin{equation}
    P\left(\|\partial^{\alpha} \mathscr{A}_n^{-1}(x)\| \leq C_{\alpha_0} h_n^{-\left|\alpha\right|} \Big| \mathscr{E}_I\right) \geq 1-\frac{1}{n^{s_1}},
    \end{equation}
    where $C_{\alpha_0},s_1$ are positive constant.
    \label{lemma 3.10 davoud - bound a^-1}
\end{lemma}
 
\begin{proof}

Applying the multivariate Leibniz rule to the identity
$$
\mathscr{A}_n^{-1}(x) \mathscr{A}_n(x)=I .
$$
we obtain the following recursive formula for the derivatives of the inverse matrix (see, e.g., \cite{hormander1983analysis}), when we differentiating both sides with respect to a multi-index $\alpha_0$ and evaluating at a point $x \in \Omega$,

$$
\partial^{\alpha_0} \mathscr{A}_n^{-1}(x)=-\sum_{\beta<\alpha_0}\binom{\alpha_0}{\beta} \mathscr{A}_n^{-1}(x)\left(\partial^{\alpha_0-\beta} \mathscr{A}_n(x)\right)\left(\partial^\beta \mathscr{A}_n^{-1}(x)\right) .
$$
We now prove the desired result by induction on $\left|\alpha_0\right|$.
Starting with the base case: $\left(\left|\alpha_0\right|=1\right)$.
Let $\alpha_0=e_j$, the unit multi-index in the $j$ th direction. 
Then,
$$
\partial^{e_j} \mathscr{A}_n^{-1}(x)=-\mathscr{A}_n^{-1}(x)\left(\partial^{e_j} \mathscr{A}_n(x)\right) \mathscr{A}_n^{-1}(x).
$$
From Lemma \ref{lemma 3.6 davoud -bound on dA} we have, 
\begin{equation}
    P\left(\|\partial^{e_j} \mathscr{A}_n(x)\|\leq \tilde{C}_{\alpha} h_n^{-1}\Big|\mathscr{E}_I\right)\geq 1-\frac{2}{n^{\tilde{c} }},
    \label{eq:result lemma 5 in lemma 8}
\end{equation}
and from Lemma \ref{lemma 3.7 davoud - lambda_min}, we know that given the event $\mathscr{E}_I$, 

\begin{equation}
P\left(\left\|\mathscr{A}_n^{-1}(x)\right\|_2 \leq \frac{1}{C_\lambda} \Big|\mathscr{E}_I\right)\geq 1-\frac{2}{n^{\phi}}.
\label{norm 2 of A_n^-1}
\end{equation}
Using the sub multiplicative property of the matrix norm, $\|AB\|\leq\|A\|\|B\|$, and the bounds from \eqref{eq:result lemma 5 in lemma 8} and \eqref{norm 2 of A_n^-1}, we have:
$$
\left\|\partial^{e_j} \mathscr{A}_n^{-1}(x)\right\| \leq\left\|\mathscr{A}_n^{-1}(x)\right\|^2 \cdot\left\|\partial^{e_j} \mathscr{A}_n(x)\right\|.
$$
By the union bound, the probability that both $\left\|\mathscr{A}_n^{-1}(x)\right\| \leq 1 / C_\lambda$ and $\left\|\partial^{e_j} \mathscr{A}_n(x)\right\| \leq \tilde{C}_\alpha h_n^{-1}$ hold is at least $1-\frac{2}{n^{\bar{c}}}-\frac{2}{n^\phi}$. Thus, we conclude:
$$
P\left(\left\|\partial^{e_j} \mathscr{A}_n^{-1}(x)\right\| \leq C_{e_j} h_n^{-1} \mid \mathscr{E}_I\right) \geq 1-\frac{2}{n^{\tilde{c}}}-\frac{2}{n^\phi},
$$
where $C_{e_j}=\tilde{C}_\alpha / C_\lambda^2$.
This establishes the base case of the induction.
Next, the inductive step
assumes that for all multi-indices $\beta<\alpha_0$, we have for all $x \in \Omega$
\begin{equation}
    P\left(\|\partial^\beta \mathscr{A}_n^{-1}(x)\|\leq C_\beta h_n^{-|\beta|}\Big| \mathscr{E}_I\right)\geq 1-\frac{2}{n^{\tilde{c} }}-\frac{2}{n^{\phi}}.
    \label{inductive assumption}
\end{equation}
Now we prove the inductive step. From the derivative identity we obtain:
\begin{equation}
   \|\partial^{\alpha_0} \mathscr{A}_n^{-1}(x)\| \leq \sum_{\beta<\alpha_0}\binom{\alpha_0}{\beta}\left\|\mathscr{A}_n^{-1}(x)\right\|_2 \cdot\|\partial^{\alpha_0-\beta} \mathscr{A}_n(x)\| \cdot\|\partial^\beta \mathscr{A}_n^{-1}(x)\| .
    \label{derivative identity}
\end{equation}
Using the induction hypothesis \eqref{inductive assumption} together with Equations \eqref{norm 2 of A_n^-1},\eqref{derivative identity}, and the bounds from Lemma \ref{lemma 3.6 davoud -bound on dA} for $\alpha_0-\beta$, we obtain that,
$$
P\left(|\partial^{\alpha_0} \mathscr{A}_n^{-1}({x})|\leq C_{\alpha_0} h_n^{-|\alpha_0|}\Big|\mathscr{E}_I\right)\geq 1-\frac{2}{n^{\tilde{c}}}-\frac{2}{n^{\phi}},
$$
where $C_{\alpha_0}=\frac{1}{C_\lambda} \sum_{\beta<\alpha_0}\binom{\alpha_0}{\beta} \tilde{C}_{\alpha_0-\beta} C_\beta$.
This completes the inductive argument, as desired. 
By taking a positive constant  $s_1<\min\{\tilde{c},\phi\}$, the constant factor in the tail bound can be absorbed for $n>1$. Thus,
$$
P\left(|\partial^{\alpha_0} \mathscr{A}_n^{-1}({x})|\leq C_{\alpha_0} h_n^{-|\alpha_0|}\Big|\mathscr{E}_I\right)\geq 1-\frac{1}{n^{s_1}},
$$
which completes the proof. 
\end{proof}

Finally, we use the preceding bounds to control the derivatives of the shape functions $a_i^*(x)$ themselves. These bounds ensure that the MLS reconstruction remains uniformly well-conditioned across the domain $\Omega$, providing the stability required to guarantee the convergence of the approximation. 
The following Lemma provides the stochastic counterpart to Theorem 3.11 in \cite{mirzaei2015analysis}.

\begin{lemma}
Suppose the event $\mathscr{E}_I$ as defined in \eqref{E_I event} holds, and all the assumptions of Lemma \ref{bound of weight fun} are satisfied. Then, for $|\alpha|\leq k-1$ and a fixed $x\in \Omega$, the MLS shape functions $a^*_i$ defined in \eqref{explicitly a*_i} satisfy the following stability condition,
\begin{equation}
    P\left(\sum_{i \in I_{B\left(x, s h_n\right)}}\left|\partial^\alpha a_i^*(x)\right| \leq C_\alpha^{\prime} \cdot h_n^{-|\alpha|} \Big|\mathscr{E}_I\right)\geq 1-\frac{1}{n^{s_1}},
\end{equation}
where $C'_\alpha,s_1$ are positive constant.
\label{lemma 3.11 davoud-bound shape function}
\end{lemma}

\begin{proof}
Recall the stochastic MLS shape functions \eqref{explicitly a*_i} are given by
\begin{equation}
     a_i^*(x)=\theta_h(x_i-x) \sum_{|\beta| \leq k-1} \eta_\beta(x) \cdot \frac{\left(x_i-x\right)^\beta}{h_n^{|\beta|}} .
     \label{a_i^* in lemma 3.11}
\end{equation}
We start with bounding the derivatives of $\eta(x)=\dfrac{1}{N_{B(\hat{x}, s h_n)}}\,\mathscr{A}_n(x)^{-1} \boldsymbol{p}(x)$ \eqref{eta normalized}, where the count $N_{B(\hat{x}, s h_n)}$ does not depend on the differentiation variable $x$ (it is the count at the fixed center $\hat{x}$), and is therefore carried as a constant scalar under $\partial^\alpha$.
\begin{equation*}
    \partial^\alpha \eta(x)=\frac{1}{N_{B(\hat{x}, s h_n)}}\sum_{\zeta \leq \alpha}\binom{\alpha}{\zeta} \partial^{\alpha-\zeta} \mathscr{A}_n^{-1}(x) \cdot \partial^\zeta \boldsymbol{p}(x) \quad \forall \alpha \text { with }|\alpha| \leqslant k-1 .
\end{equation*}
Now, by the definition of the scaled polynomial basis \eqref{eq:polynomial basis}, we have:
$$
\partial^\zeta \boldsymbol{p}(x)=\zeta!\cdot h_n^{-|\zeta|} \cdot e_\zeta,
$$
where $e_\zeta$ is the canonical basis vector in multi-index ordering (i.e., the standard basis vector with a 1 in the position corresponding to multi-index $\zeta$, and zeros elsewhere).
By Lemma \ref{lemma 3.10 davoud - bound a^-1}, with probability of at least $1-\frac{1}{n^{s_1 }}$, we have
$$
\left|\partial^\gamma \mathscr{A}_n^{-1}(x)\right| \leq C_\gamma h_n^{-|\gamma|},
$$
where $C'_\gamma$ is a positive constant independent of $x$ and $h_n$. 
Hence, with the same probability, for a constant $C_{\alpha,\zeta}$:
\begin{equation}
    \left|\partial^\alpha \eta(x)\right| \leq \frac{1}{N_{B(\hat{x}, s h_n)}}\sum_{\zeta \leq \alpha}\binom{\alpha}{\zeta} {C_{\alpha, \zeta} h_n^{-|\alpha-\zeta|}} \cdot e_\zeta h_n^{-|\zeta|} \leq \frac{C_{\alpha}}{N_{B(\hat{x}, s h_n)}} h_n^{-|\alpha|} ,
    \label{eq:bound on eta}
\end{equation}
where the vector $C_\alpha\in\mathbb{R}^{|A|}$ is a bound for $\sum_{\zeta \leqslant \alpha}\binom{\alpha}{\zeta} C_{\alpha, \zeta} e_\zeta$.
Taking the derivatives of both sides in Equation \eqref{a_i^* in lemma 3.11} and using the product rule, we have,
\begin{equation*}
    \partial^\alpha a_i^*(x)=\sum_{\zeta \leq \alpha}\binom{\alpha}{\zeta} \partial^{\alpha-\zeta} \theta_h(x_i-x) \sum_{|\beta| \leq k-1} \partial^\zeta \left(\eta_\beta(x) \cdot \frac{\left(x_i-x\right)^\beta}{h_n^{|\beta|}}\right) .
\end{equation*}
By Lemma \ref{bound of weight fun}, the weight functions satisfy:
$$
\left|\partial^\alpha \theta_{h}(x_i-x)\right| \leq C_{\alpha,\theta} h_n^{-|\alpha|},
$$
and since  $x_i \in B\left(x, s h_n\right)$, we have $\left|\left(x_i-x\right)^\beta\right| \leq\left(s h_n\right)^{|\beta|}$. Applying the latter in addition to the previous bounds, we get that with probability of at least $1-\frac{1}{n^{s_1 }}$,
\begin{equation*}
    \left|\partial^\alpha a_i^*(x)\right|  \leqslant \sum_{\zeta \leqslant \alpha}\left\{\binom{\alpha}{\zeta} C_{\alpha, \zeta} h_n^{|\zeta|-|\alpha|} \sum_{|\beta| \leqslant k-1} \frac{C_\alpha}{N_{B(\hat{x}, s h_n)}} h_n^{-|\zeta|} h_n^{-|\beta|} h_n^{|\beta|}\right\} \leqslant \frac{C_\alpha}{N_{B(\hat{x}, s h_n)}} h_n^{-|\alpha|}.
\end{equation*}
Summing over the active indices and using that the number of nonzero terms is exactly $\#I_{B(\hat{x}, s h_n)}=N_{B(\hat{x}, s h_n)}$, the factor $N_{B(\hat{x}, s h_n)}$ cancels the $1/N_{B(\hat{x}, s h_n)}$ in the per--term bound:
$$
\sum_{i \in I_{B(\hat{x}, s h_n)}}\left|\partial^\alpha a_i^*(x)\right|
\;\leq\; N_{B(\hat{x}, s h_n)}\cdot \frac{C_\alpha}{N_{B(\hat{x}, s h_n)}}\, h_n^{-|\alpha|}
\;=\; C_\alpha\, h_n^{-|\alpha|}.
$$
Hence, on the event $\mathscr{E}_I$ (which holds with probability at least $1-\tfrac{2}{n^c}$ by Lemma \ref{convergence rate of sample size growth}), for all $x\in\Omega$,
\begin{equation}
    P\left(\sum_{i \in I_{B\left(x, s h_n\right)}}\left|\partial^\alpha a_i^*(x)\right| \leq C_\alpha^{\prime} \cdot h_n^{-|\alpha|} \Big|\mathscr{E}_I\right)\geq 1-\frac{1}{n^{s_1 }},
    \label{eq: point-wise lemma 9}
\end{equation}
 where $C'_\alpha$ is a positive constant independent of $x$ and $h_n$, as desired.
 \end{proof}

\section{Smoothness of the Approximation} \label{smoothness of the approximation}

In this section, we show that the MLS approximation $\mlsfunstoch(x)$ inherits smoothness locally from the weight function and from the regularity of the local design matrices. This is formally proven below, but first, we present the intuition behind the proof.

The structural assumptions on the point cloud ensure that, with high probability, each local neighborhood contains a sufficiently dense set of points to uniquely determine the polynomial coefficients. This guarantees that, on a sufficiently small ball of radius comparable to $h_n$, the local Gram matrix $\mathscr{G}_n(x)$ is full-rank and varies $\mathcal{C}^k$-smoothly with respect to $x$ --- unlike its normalization $\mathscr{A}_n(x)=\frac{1}{N_{B(x,h_n)}}\mathscr{G}_n(x)$, whose sample count $N_{B(x,h_n)}$ jumps with $x$. Since the weight function $\theta_h$ belongs to $\mathcal{C}^k$ and $\mathscr{G}_n(x)$ is smoothly invertible there, the shape functions $a_i^*(x)$ are likewise of class $\mathcal{C}^k$ on this local ball.

Consequently, the MLS approximation $\mlsfunstoch(x)$, expressed locally as a linear combination of the sampled values $f(x_i)$ with $\mathcal{C}^k$ shape functions, belongs to $\mathcal{C}^k$ on a ball whose radius is of order $h_n$. 

\begin{proof}[Proof of Theorem~\ref{thm:local-smoothness}]
The key is to express the shape functions through the \emph{unnormalized} Gram matrix
$\mathscr{G}_n$ of \eqref{gram system}, which---unlike the normalized matrix
$\mathscr{A}_n=\tfrac{1}{N_{B(x,h_n)}}\mathscr{G}_n$, depends smoothly on $x$. By
\eqref{eta normalized}, $\eta(x)=\mathscr{G}_n(x)^{-1}\boldsymbol{p}(x)$, hence
\begin{equation}
a_i^*(x)=\theta_h(x_i-x)\,\boldsymbol{p}_x(x_i)^\top\,\mathscr{G}_n(x)^{-1}\,\boldsymbol{p}(x),
\qquad
\boldsymbol{p}_x(x_i)=\Big(\tfrac{(x_i-x)^\gamma}{h_n^{|\gamma|}}\Big)_{|\gamma|\le k-1}.
\label{eq:shape-via-gram}
\end{equation}
The sample count $N_{B(x,h_n)}$ has cancelled and does \emph{not} appear in \eqref{eq:shape-via-gram}.
We now show the smoothness of $\mathscr{G}_n$.
By \eqref{A_nab in main proof},
$$
\mathscr{G}_{n\,ij}(x)=\sum_{t=1}^n \theta_h(x_t-x)\,
\frac{(x_t-x)^{\alpha_i}}{h_n^{|\alpha_i|}}\,\frac{(x_t-x)^{\alpha_j}}{h_n^{|\alpha_j|}},
\qquad i,j=1,\dots,A .
$$
Each summand is the product of the $\mathcal{C}^k$ weight $\theta_h(x_t-x)$ (which
vanishes to order $k$ at the boundary of its support) with polynomials, hence is
$\mathcal{C}^k$ in $x$; a sample $x_t$ entering or leaving the support $B(x,sh_n)$ does so
with $\theta_h\to 0$ smoothly. Thus $\mathscr{G}_{n\,ij}\in\mathcal{C}^k(\Omega)$. In
contrast to $\mathscr{A}_n$, the Gram matrix carries no count normalization, so it does not inherit the discontinuity of $x\mapsto N_{B(x,h_n)}$.

Now, Fix $x_0\in\Omega$ and work on the event $\mathscr{E}_I$ throughout. By
Lemma~\ref{lemma 3.7 davoud - lambda_min}, with conditional probability at least
$1-\tfrac{2A^2}{n^{\phi}}$ we have $\lambda_{\min}(\mathscr{A}_n(x_0))\ge C_\lambda$;
on this event, since $\mathscr{G}_n(x_0)=N_{B(x_0,sh_n)}\,\mathscr{A}_n(x_0)$ and, by
$\mathscr{E}_I$, $N_{B(x_0,sh_n)}\ge\gamma_1\log n$,
\begin{equation}
\lambda_{\min}(\mathscr{G}_n(x_0))\;\ge\;\gamma_1 C_\lambda\,\log n .
\label{eq:gram-lambda-min}
\end{equation}
To propagate this to a ball, we bound the Lipschitz constant of the entries of
$\mathscr{G}_n$. Any sample contributing to $\mathscr{G}_n(x)$ for some $x\in B(x_0,\delta_0)$
with $\delta_0\le sh_n$ lies within $2sh_n$ of $x_0$; 
the number of such samples is at most the count in $B(x_0,2sh_n)$, which by
Lemma~\ref{convergence rate of sample size growth} (applied with radius $2sh_n$) is at most
$\gamma_2'\log n$ for a constant $\gamma_2'$; on $\mathscr{E}_I$ this holds together with the
lower bound \eqref{eq:gram-lambda-min}.
Differentiating an entry brings one
factor $h_n^{-1}$ through $\theta_h$, so there is a constant $C_{\mathrm{Lip}}$,
\emph{independent of $n$}, with
$$
\sum_{i,j}\big|\mathscr{G}_{n\,ij}(x_0)-\mathscr{G}_{n\,ij}(x)\big|
\;\le\;C_{\mathrm{Lip}}\,d\,(\log n)\,h_n^{-1}\,\|x_0-x\| .
$$
With Weyl's inequality and the norm equivalence $\|B\|_{\mathrm{op}}\le C_{\mathrm{norm}}\|B\|_{F,1}$,
$$
\big|\lambda_{\min}(\mathscr{G}_n(x_0))-\lambda_{\min}(\mathscr{G}_n(x))\big|
\;\le\;C_{\mathrm{norm}}C_{\mathrm{Lip}}\,d\,(\log n)\,h_n^{-1}\,\|x_0-x\| .
$$
Set
$$
\delta_0:=\frac{\gamma_1 C_\lambda\,\log n}{2\,C_{\mathrm{norm}}C_{\mathrm{Lip}}\,d\,(\log n)\,h_n^{-1}}
=\frac{\gamma_1 C_\lambda}{2\,C_{\mathrm{norm}}C_{\mathrm{Lip}}\,d}\,h_n\;\asymp\;h_n .
$$
The two factors of $\log n$ cancel, leaving $\delta_0\asymp h_n$. 

For every $x\in B(x_0,\delta_0)$, the Lipschitz bound together with $\|x_0-x\|\le\delta_0$ and the definition of $\delta_0$ gives
\[
\big|\lambda_{\min}(\mathscr{G}_n(x_0))-\lambda_{\min}(\mathscr{G}_n(x))\big|
\;\le\; C_{\mathrm{norm}}C_{\mathrm{Lip}}\,d\,(\log n)\,h_n^{-1}\,\delta_0
\;=\;\tfrac12\,\gamma_1 C_\lambda\log n
\;\le\;\tfrac12\,\lambda_{\min}(\mathscr{G}_n(x_0)),
\]
where the last inequality uses \eqref{eq:gram-lambda-min}. Hence
\[
\lambda_{\min}(\mathscr{G}_n(x))\;\ge\;\lambda_{\min}(\mathscr{G}_n(x_0))-\tfrac12\,\lambda_{\min}(\mathscr{G}_n(x_0))
\;=\;\tfrac12\,\lambda_{\min}(\mathscr{G}_n(x_0))\;>\;0,
\]
so $\mathscr{G}_n(x)$ is positive definite throughout $B(x_0,\delta_0)$.

Therefore, on $B(x_0,\delta_0)$ the matrix $\mathscr{G}_n(x)$ is $\mathcal{C}^k$ and invertible; by
Cramer's rule its inverse $\mathscr{G}_n(x)^{-1}$ is $\mathcal{C}^k$ there (entries are
$\mathcal{C}^k$ cofactors divided by $\det\mathscr{G}_n(x)$, which is bounded away from $0$).
Hence $\eta(x)=\mathscr{G}_n(x)^{-1}\boldsymbol{p}(x)\in\mathcal{C}^k$, and by
\eqref{eq:shape-via-gram} with $\theta_h\in\mathcal{C}^k$, each shape function
$a_i^*(x)\in\mathcal{C}^k(B(x_0,\delta_0))$. Therefore
$$
P\!\left(\mlsfunstoch(x)\in\mathcal{C}^k\big(B(x_0,\delta_0)\big)\,\middle|\,\mathscr{E}_I\right)
\;\ge\;1-\frac{2A^2}{n^{\phi}} .
$$
\end{proof}

\begin{coro}[Local version]\label{cor:local-shapefun}
Given the event $\mathscr{E}_I$ holds, with conditional probability at least
$1 - \tfrac{2A^2}{n^{\phi}}$ (where $\phi$ is the constant from
Lemma~\ref{lemma 3.7 davoud - lambda_min}), each shape function $a^*_i(x)$ defined in
\eqref{explicitly a*_i} belongs to $\mathcal{C}^k\!\big(B(x_0,\delta_0)\big)$. Moreover, with the
same probability, for every multi-index $\alpha$ with $|\alpha| \le k$, the partial derivative
$\partial^\alpha a^*_i(x)$ is uniformly continuous on $B(x_0,\delta_0)$.
\end{coro}

\begin{proof}
The proof of Theorem~\ref{thm:local-smoothness} established that, on the event $\mathscr{E}_I$ and
with conditional probability at least $1 - \tfrac{2A^2}{n^{\phi}}$, the Gram matrix
$\mathscr{G}_n(x)$ is $\mathcal{C}^k$ and positive definite for all $x \in B(x_0,\delta_0)$,
\begin{equation}\label{eq:lambdamin-ball}
  \lambda_{\min}\big(\mathscr{G}_n(x)\big) \;\ge\; \tfrac12\,\lambda_{\min}\big(\mathscr{G}_n(x_0)\big)
  \;>\;0, \qquad \text{for all } x \in B(x_0,\delta_0),
\end{equation}
with $\det\mathscr{G}_n(x)$ bounded away from $0$. By Cramer's rule the entries of
$\mathscr{G}_n(x)^{-1}$ are then in $\mathcal{C}^k$ on $B(x_0,\delta_0)$, so the coefficients
$\eta(x) = \mathscr{G}_n(x)^{-1}\boldsymbol{p}(x)$ are in $\mathcal{C}^k$; together with the smoothness
of the weight $\theta_h$ and \eqref{eq:shape-via-gram}, each shape function $a^*_i(x)$ is in
$\mathcal{C}^k\!\big(B(x_0,\delta_0)\big)$.

For the second claim, fix a multi-index $\alpha$ with $|\alpha| \le k$. Since
$a^*_i \in \mathcal{C}^k\!\big(B(x_0,\delta_0)\big)$, the partial derivative $\partial^\alpha a^*_i$
is continuous on $B(x_0,\delta_0)$. The closed ball ${B(x_0,\delta_0)}$ is compact, and a
continuous function on a compact set is uniformly continuous; hence $\partial^\alpha a^*_i$ is
uniformly continuous on $B(x_0,\delta_0)$, with the same conditional probability
$1 - \tfrac{2A^2}{n^{\phi}}$.
\end{proof}

\begin{remark}[Globalization and its cost]\label{rem:global-cost}
The locality of Theorem~\ref{thm:local-smoothness} and
Corollary~\ref{cor:local-shapefun} stems from the radius $\delta_0 \asymp h_n$,
which in turn reflects the $h_n^{-1}$ scaling of the Lipschitz constant of the
entries of $\mathscr{A}_n$. Choosing a bandwidth of the order used by Stone
\cite{stone1982optimal} would enlarge these balls enough to cover $\Omega$ and
would yield global convergence rates in the spirit of his results. The resulting
rate, however, is constrained jointly by the Lipschitz constant and the
bandwidth, and is therefore suboptimal: a global rate is achievable, but at a
quantifiable price in approximation error.
\end{remark}
 
\begin{remark}[Why the trade-off is benign for noisy data but not here]\label{rem:noise}
This price is one that Stone's setting absorbs at no additional cost. Had Stone
employed a smooth weight function, he would plausibly retain global continuity at
the same rates, because his convergence rate is already limited by the noise in
his samples; enlarging the bandwidth would not be expected to degrade the rate
beyond the loss the noise has already imposed, which is what permits a global
statement. Our setting is different: we sample in the noise-free regime, where the
rate is governed by the fill distance. Here any bandwidth larger than the fill
distance is a pure loss, with no pre-existing noise floor to absorb it. We
therefore forgo the global statement deliberately, keeping the
fill-distance--scaled bandwidth and the sharp local result rather than trading it
for a suboptimal global one.
\end{remark}
 
\section{Conclusion}

In this work, we have provided a rigorous stochastic foundation for the Moving
Least Squares (MLS) method, bridging the gap between its deterministic origins in
numerical analysis and its statistical counterparts in Local Polynomial
Regression. By quantifying the asymptotic behavior of the fill distance $h_n$ and
the separation $\delta_n$ for i.i.d.\ samples, we demonstrated that the
quasi-uniformity assumption---a cornerstone of classical MLS theory---fails to
hold in the stochastic regime.

Despite the failure of these deterministic assumptions, we established that the
fundamental strengths of MLS are preserved under random sampling. We proved that
for a $k$-times smooth function, the error in estimating differential operators of
order $m$ decays at the rate of $h_n^{\,k-|m|}$ with
high probability. This result reconciles the approximation power of MLS with the
minimax optimal rates found in nonparametric statistics. Furthermore, we showed
that the resulting approximant locally inherits the $\mathcal{C}^k$ smoothness of the weight
function---on balls whose radius scales with the bandwidth, with high
probability---ensuring its utility in applications requiring high-order regularity
even when the data is sampled stochastically; the precise sense in which this
smoothness is local, and why a global statement would cost approximation accuracy,
is discussed in Remarks~\ref{rem:global-cost}--\ref{rem:noise}.

In a forthcoming companion study, we extend these results to the setting of
Stochastic Manifold Moving Least Squares (MMLS) for data sampled from unknown
Riemannian manifolds. The manifold framework introduces additional geometric and
analytical considerations---in particular, controlling the interplay between the
intrinsic sampling geometry and the local Euclidean approximation used by
MLS---that warrant a dedicated treatment, and is therefore developed in a separate
work.
\section*{Acknowledgments}

This work was made possible through the support of the Ariane de Rothschild Women Doctoral Program for outstanding female PhD students.

\section*{Data Availability Statement}
Data sharing is not applicable to this article, as no datasets were generated
or analysed during the current study. The work is purely theoretical and mathematical.

\appendix
\renewcommand{\thesection}{Appendix \Alph{section}:}
\section{Probability Notations}
\begin{definition}
    The notation 
    \begin{equation*}
        X_n=O_p\left(a_n\right) \text { as } n \rightarrow \infty
    \end{equation*}
        means that $\frac{X_n}{a_n}$ is stochastically  bounded. That is, for any ${\varepsilon_x} > 0$, there exists a finite $M > 0$ and a finite $N_x > 0$ such that
        \begin{equation*}
         P\left(\left|\frac{X_n}{a_n}\right|>M\right)<{\varepsilon_x},\,\, \forall n>N_x.
        \end{equation*}
        \label{def_of_O_p}
    \end{definition}

\begin{definition}
    The notation 
    \begin{equation*}
        Y_n=\Omega_p\left(b_n\right) \text { as } n \rightarrow \infty
    \end{equation*}
        means that $\frac{Y_n}{b_n}$ is stochastically  bounded. That is, for any ${\varepsilon_y} > 0$, there exists a finite $m > 0$ and a finite $N_y > 0$ such that
        \begin{equation*}
         P\left(\left|\frac{Y_n}{b_n}\right|<m\right)<{\varepsilon_y},\,\, \forall n>N_y.
        \end{equation*}
        \label{def_of_Omega_p}
    \end{definition} 

\section{Failure of the deterministic proof in stochastic setup} \label{Failure of the deterministic - Wendland}

In Wendland's work \cite{wendland2005scattered} (page 42), there is a review of Levin's work, which proves that  $\|\mlsfun{x}-f\|_{\Omega, \infty}<L \cdot h^{k} $. The technique used there involves bounding the error by controlling the shape functions through certain manipulations. In order to achieve this, it is necessary to bound the number of samples in \( I(x) \) by a constant. This is done by comparing the volume of the union of the small balls with the volume of the large ball, leading to the following bound:

\begin{equation}
    \# I(x) \operatorname{vol}(B(0,1)) \delta_n^d \leq \operatorname{vol}(B(0,1))(\delta_n+h_n)^d
\end{equation}
Using the quasi-uniformity property (which holds in the deterministic setup) finally leads to
\begin{equation}
    \# I(x) \leq\left(1+\frac{h_n}{\delta_n}\right)^d \leq\left(1+c_{qu}\right)^d.
\end{equation}
This implies that the number of indices in \( I(x) \) remains bounded as \( n \) tends to infinity, which contradicts Lemma  \ref{convergence rate of sample size growth}, as is the case in stochastic settings. Mirzaei used this result in his paper to prove his theory. He relies on it in several of his lemmas, and therefore his Theorem 3.12—which is the main result we address in our paper (Theorem \ref{main_result})—cannot be reproduced in the stochastic setting. We re-proved the claims that depend on this result without relying on the quasi-uniformity property, and together with the remaining claims that do not depend on it, we constructed the entire proof.

\section{Detailed computation of the weighted covariance matrix in proof of Theorem \ref{main_result}} \label{Normal Equations Appendix}

Recall our basis is defined as:
\[
p_{\alpha}= 
\left\{ \left(\frac{x - z}{h_n}\right)^\alpha\right\}_{\alpha\in A},
\]
Using this basis, we construct the normal equations accordingly. The design matrix takes the form
\[
\tilde{\mathscr{X}} = \left(
\begin{array}{cccc}
\sqrt{\theta_1} p_{\alpha_0}(X_1) & \sqrt{\theta_1} p_{\alpha_1}(X_1) & \cdots & \sqrt{\theta_1} p_{\alpha_{|A|}}(X_1) \\
\vdots & & & \vdots \\
\sqrt{\theta_n} p_{\alpha_0}(X_n) & \sqrt{\theta_n} p_{\alpha_1}(X_n) & \cdots & \sqrt{\theta_n} p_{\alpha_{|A|}}(X_n)
\end{array}
\right),
\]
where $\theta_i=\theta_h(xi-x)$.
From this, the normal equations are obtained as
\[
C \cdot \tilde{\mathscr{X}}^\top \tilde{\mathscr{X}} \beta = C \cdot  \tilde{\mathscr{X}}^\top F,
\]
which leads to the solution
\[
\beta = (\tilde{\mathscr{X}}^\top \tilde{\mathscr{X}})^{-1} \tilde{\mathscr{X}}^\top F.
\]
To explicitly compute $\tilde{\mathscr{X}}^{\top} \tilde{\mathscr{X}}$, we expand the expression: 
\begin{align*}
\tilde{\mathscr{X}}^\top \tilde{\mathscr{X}}  = 
\left(
\begin{array}{ccc}
\sqrt{\theta_1} p_{\alpha_0}(X_1) & \cdots & \sqrt{\theta_n} p_{\alpha_{0}}(X_n) \\
\vdots &  & \vdots \\
\sqrt{\theta_1} p_{\alpha_{|A|}}(X_1) & \cdots & \sqrt{\theta_n} p_{\alpha_{|A|}}(X_n) 
\end{array}
\right)\cdot 
\left(
\begin{array}{ccc}
\sqrt{\theta_1} p_{\alpha_0}(X_1)  & \cdots & \sqrt{\theta_1} p_{\alpha_{|A|}}(X_1) \\
\vdots & & \vdots \\
\sqrt{\theta_n} p_{\alpha_0}(X_n)  & \cdots & \sqrt{\theta_n} p_{\alpha_{|A|}}(X_n) 
\end{array}
\right)
\end{align*}
\[
=
\left(
\begin{array}{ccc}
\sum_{i}{\theta_i} p_{\alpha_0}(X_i) \cdot p_{\alpha_0}(X_i)  & \cdots & \sum_i {\theta_i} p_{\alpha_0}(X_i) \cdot p_{\alpha_{|A|}}(X_i) \\
\vdots & & \vdots \\
\sum_i {\theta_i} p_{\alpha_{|A|}}(X_i) \cdot p_{\alpha_0}(X_i)  & \cdots & \sum_i {\theta_i} p_{\alpha_{|A|}}(X_i) p_{|A|}(X_i) 
\end{array}
\right).
\]
Next, we introduce the normalization by defining $N=N_{B\left(x, h_n\right)} $ and rescaling the design matrix: 
\[
\mathscr{X} = \frac{1}{N} \tilde{\mathscr{X}},
\]
Using this notation, we express $A$ as
\[
A = \frac{1}{N} \cdot \tilde{\mathscr{X}}^\top \tilde{\mathscr{X}} = \frac{1}{N} \cdot N^2 \cdot \mathscr{X}^\top \mathscr{X}  = N \cdot \mathscr{X}^\top \mathscr{X}.
\]
With these definitions, we recover the terms $\mathscr{X}_n(x)$ and $\mathscr{A}_n(x)$  introduced in section 3. 
Finally, regarding the term $Q \mlsfunstoch$, we note that it is a scalar quantity. To see this, we examine its components:

\begin{enumerate}
    \item $\mathscr{Q}_n^{\top}$ :
 $$
\mathscr{Q}_n^{\top}=\left[\frac{\alpha_{1}!q_{\alpha_1}}{h_n^{\left[\alpha_1\right]}}, \frac{\alpha_{2}!q_{\alpha_2}}{h_n^{\left[\alpha_2\right]}}, \ldots, \frac{\alpha_{\#(\mathscr{N})}!q_{\alpha_{\#(\mathscr{N})}}}{h_n^{\left[\alpha_{\#(\mathscr{N})}\right]}}\right].
$$
   \item $\mathscr{A}_n^{-1}(x)$ :
General elements of $\mathscr{A}_n^{-1}(x)$ are denoted by $\mathscr{A}_{n\alpha \beta}^{-1}(x)$, have the form of $\tilde{\mathscr{X}}^\top \tilde{\mathscr{X}}$ up to multiplication by a general constant.\\
\item $\mathscr{X}_n^{\top}(x)$ :
The transpose of $\mathscr{X}_n(x)$, where:
$$
\mathscr{X}_{n i \alpha}(x)= \begin{cases}\frac{\sqrt{\theta_h(x_i-x)}}{N_{B\left(x, h_n\right)}} \frac{\left(x_i-x\right)^\alpha}{h_n^{|a|}}, & i \in I_{B\left(x, h_n\right)} \text { and } \alpha \in \mathscr{N}, \\ 0, &  i \notin I_{B\left(x, h_n\right)} \text { and } \alpha \in \mathscr{N}.\end{cases}
$$
\end{enumerate}
By multiplying these three terms, we obtain 
\begin{equation}
 \sum_{\alpha \in A} \sum_{\beta \in A} \sum_{ i \in I_{B\left(x, h_n\right)}} \frac{1}{ N_{B\left(x, h_n\right)}}\cdot\left[\mathscr{A}_n^{-1}(x)\right]_{\alpha \beta}\cdot \frac{\left(x_i-x\right)^\beta}{h_n^{|\beta|}} \cdot\frac{\alpha!q_\alpha}{h_n^{|\alpha|}}\cdot\left(F_n(x)\right).
\end{equation}

\bibliographystyle{abbrv}
\bibliography{biblo}

\end{document}